\documentclass[10pt]{amsproc}

\usepackage{subfig}
\usepackage{graphicx}

\usepackage{tikz}
\usetikzlibrary{calc}

\usepackage{amsmath}
\usepackage{amsthm}
\usepackage{amsfonts}
\usepackage{amssymb}
\usepackage{bigints}
\usepackage{mathtools}
\usepackage{microtype}
\usepackage{float}
\usepackage{verbatim}
\usepackage{url}
\usepackage{sidecap}
\usepackage[font={small}]{caption}

\DeclarePairedDelimiter\abs{\lvert}{\rvert}%
\DeclarePairedDelimiter\norm{\lVert}{\rVert}%

\makeatletter
\let\oldabs\abs
\def\abs{\@ifstar{\oldabs}{\oldabs*}}
\let\oldnorm\norm
\def\norm{\@ifstar{\oldnorm}{\oldnorm*}}
\makeatother

\newcommand{\sign}{\text{sgn}}

\copyrightinfo{}{}

\newtheorem{theorem}{Theorem}[section]
\newtheorem{lemma}[theorem]{Lemma}
\newtheorem{prep}[theorem]{Proposition}
\newtheorem{cor}{Corollary}

\theoremstyle{definition}

\theoremstyle{remark}
\newtheorem{remark}[theorem]{Remark}
\newtheorem{obs}[theorem]{Observation}

\theoremstyle{definition}

\numberwithin{equation}{section}

\begin{document}

\title[Propagation of curved folding]{Propagation of curved folding: The folded annulus with multiple creases exists}

\author[L. Alese]{Leonardo Alese}

\address{TU Graz \\ Department of Mathematics \\ Institute of Geometry \\
Kopernikusgasse 24 \\ 8010 Graz \\ Austria} 

\email{alese@tugraz.at}



\begin{abstract} 
In this paper we consider developable surfaces which are isometric to planar domains and which are piecewise differentiable, exhibiting folds along curves. The paper revolves around the longstanding problem of
existence of the so-called folded annulus with multiple creases, which we partially settle by building upon a deeper understanding of how a curved fold propagates to additional prescribed foldlines. After recalling some crucial properties of developables, we describe the local behaviour of curved folding employing normal curvature and relative torsion as parameters and then compute the very general relation between such geometric descriptors at consecutive folds, obtaining novel formulae enjoying a nice degree of symmetry. We make use of these formulae to prove that any proper fold can be propagated to an arbitrary finite number of rescaled copies of the first foldline and to give reasons why problems involving infinitely many foldlines are harder to solve. 
\end{abstract}

\maketitle

\section{Introduction} In recent years growing attention has been paid to the field of mathematical origami. The process of folding paper with the intent of crafting objects of art dates back to ancient China and Japan; although the earliest hard evidence of such an exercise is from the 16th century, it is possible that paper folding has been already practiced shortly after paper arrived in Japan via Buddhist monks in the 6th century \cite{lang1988}. As objects of combinatorics and kinematics, origami have been studied by many authors over a broad and diverse literature \cite{linkages}, \cite{turner}.

Moving from the seminal paper \cite{huffman1976}, the scientific community has also investigated the differential geometry of origami obtained by folding along curves, rather than straight lines. This is due not just to a theoretical interest but also to the role that surfaces obtainable by bending a flat foil (developables) have acquired in the interdependent fields of design, manufacturing and architecture in recent years \cite{review}, \cite{kilian}, \cite{freeform}, \cite{corinthians}, \cite{cladding}, \cite{shelden}.     

Even if the local geometry of folding along a single curve is well understood \cite{more}, the case of a nontrivial pattern of foldlines is challenging and may require ad hoc solutions \cite{demainelens} or numerical optimization \cite{2019curve}. The main intention of the present paper is to approach the propagation of a curved fold to the next prescribed foldline from a broad perspective, highlighting the role played by the regression curve of developables and providing formulae that describe the phenomenon in its full generality and complexity but that can still be employed to get new insights on its specificities. Also, we want to address the well known problem concerning the foldability of patterns involving concentric closed convex foldlines and contribute to the issue raised at the very end of \cite{nonexistence}:
\begin{quote}
\dots\ we conjecture that the circular pleat indeed folds, and that so too does any similar crease pattern consisting of a concentric series of convex smooth curves. Unfortunately a proof remains elusive. Such a proof would be the first proof to our knowledge of the existence of any curved-crease origami model, beyond the local neighborhood of a single crease.
\end{quote}
Some existence results were obtained \cite{demainelens} but to the knowledge of the author no progress has been made in constructing examples of folds along multiple concentric curves. We here finally provide explicit instances of such a kind (Fig.\ \ref{hypar3povray}, \ref{torus4povray}), present arguments that guarantee the existence of folds involving any finite number of concentric foldlines and give reasons why the proof still remains \textit{elusive} when it comes to patterns with infinitely many foldlines. We want to stress that in this paper we tackle the curved folding subject from the perspective of isometric maps, without addressing the issue of continuous deformations, which is nevertheless another interesting and relevant topic. In our setting, folds as the one in Fig.\ \ref{torus4povray} are legitimate while they would not be possible if one requires the existence of a continuous deformation: in our example the linking number of any two curves bounding a developable strip, which is invariant under isotopy, is different from the linking number of two concentric circles \cite{rolfsen}.       

\begin{figure}[ht!]
\scalebox{0.9}{\includegraphics{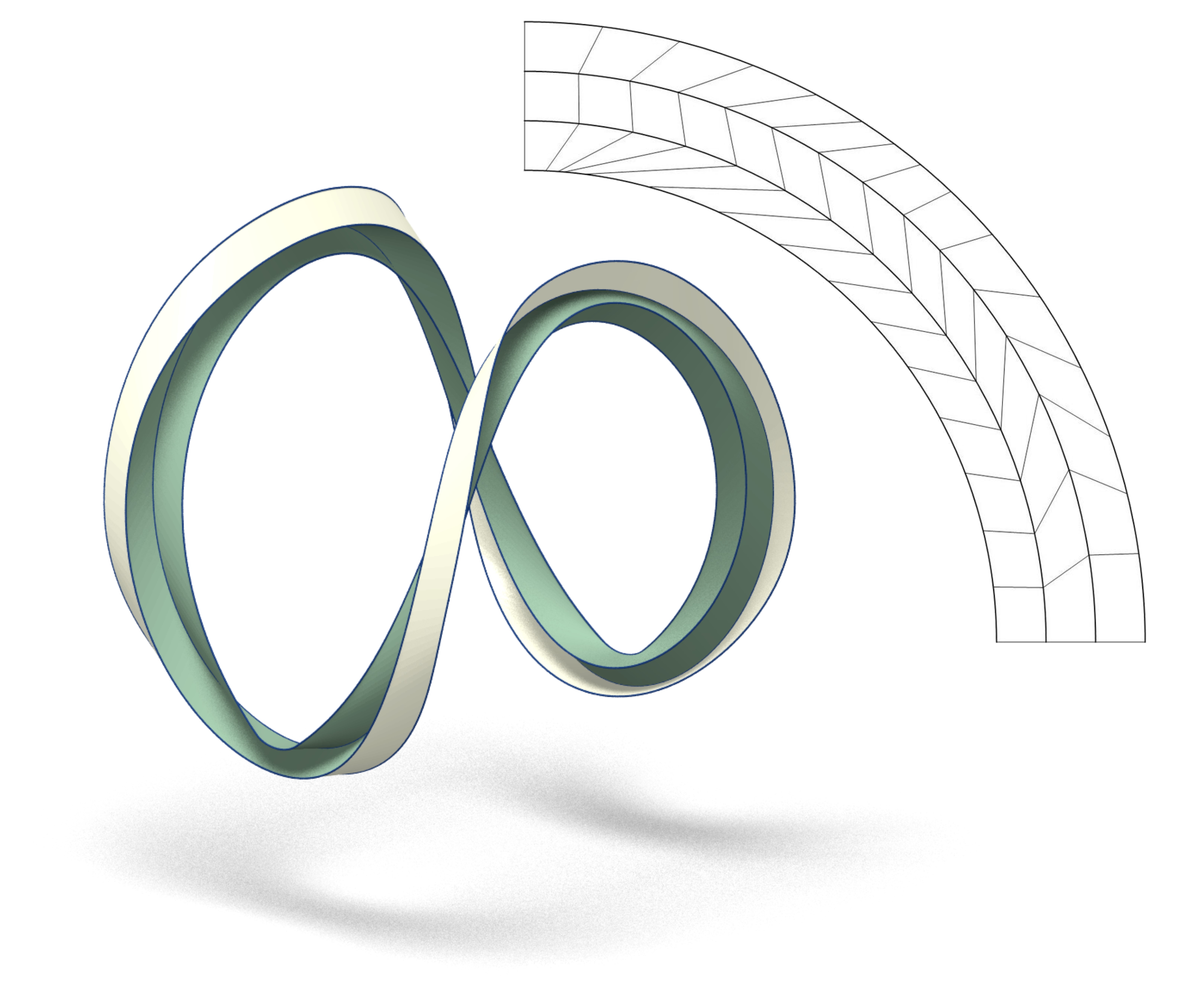}} \caption{Fold along two concentric circles of an annulus with inner radius $0.905$ and outer radius $1.19$. The two inner developables (green) are obtained by extending the isometry between the unit circle and the rescaled intersection of the unit sphere with the hyperbolic paraboloid of equation $z-3 xy=0$. The outer developable (white) is induced once the second concentric foldline is prescribed. On the right, we show the ruled structure of one of the circular sectors of the annulus (equivalent up to reflection). } \label{hypar3povray}  
\end{figure} 

\begin{figure}[ht!]
\scalebox{0.8}{\includegraphics{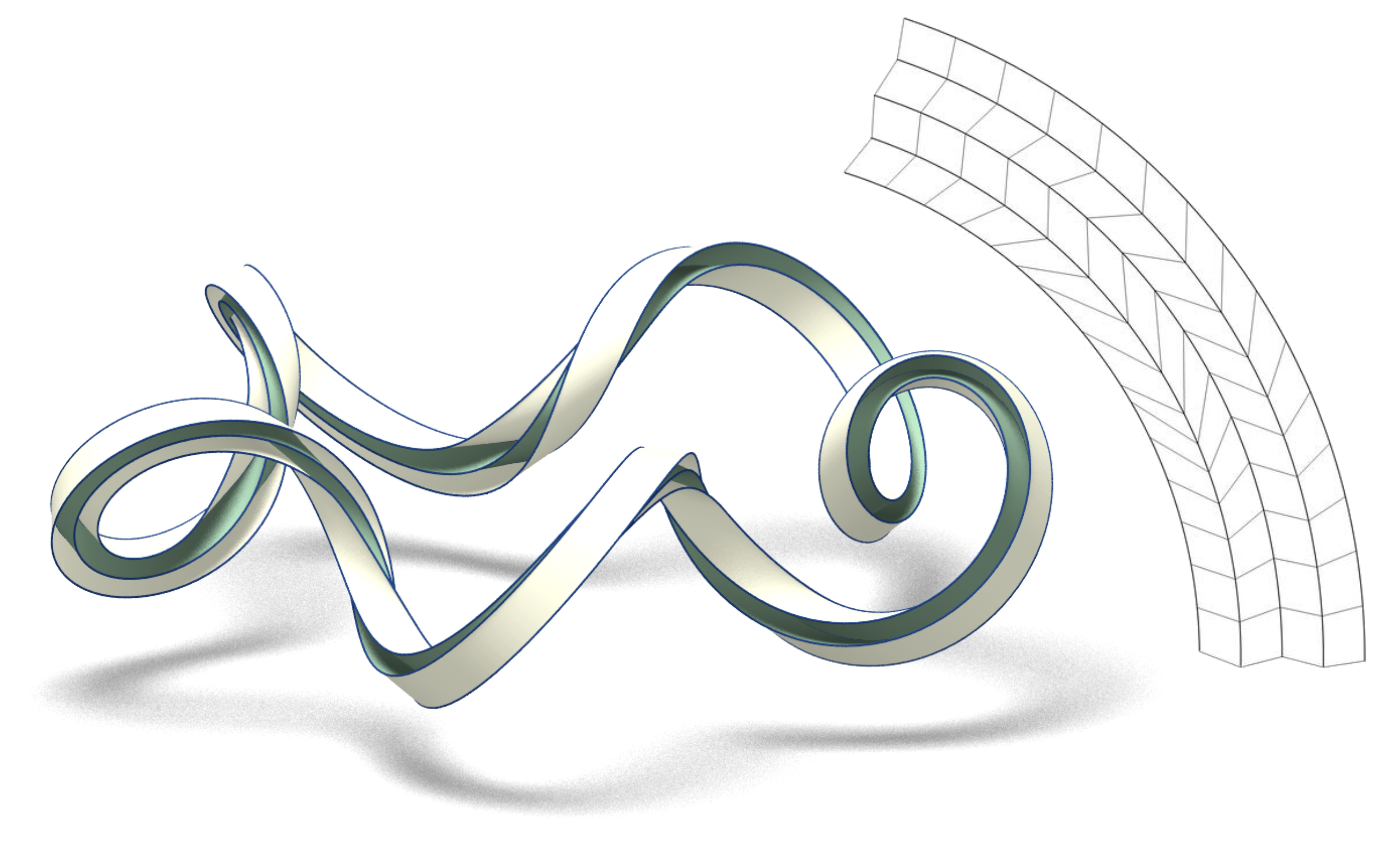}} \caption{Fold along three concentric circles of an annulus with inner radius $0.86$ and outer radius $1.14$. The two middle developables (green) are obtained by extending the isometry between the unit circle and the toroidal unknot $\omega_{3,(9,2)}$ (see \S \ref{localconvex}). The outer and the inner developables (white) are induced once two additional concentric foldlines are prescribed. On the right, we show the ruled structure of one of the circular sectors of the annulus (equivalent up to rotation).} \label{torus4povray}  
\end{figure} 

As for an outline of the content, \S \ref{recall} settles the notation about some natural geometric descriptors for curves and surfaces of Euclidean space and recalls how a surface isometric to the plane admits a ruled parametrization. In \S \ref{local} we describe how paper locally folds along a curve by discussing its behaviour in terms of the normal curvature and relative torsion of the ridge; the degree of symmetry of the formulae obtained points out how such parameters are to some extent the natural ones to describe the problem. In \S \ref{localconvex}, two methods for folding along a circle and, more in general, along a closed convex curve are described. In \S \ref{secprop1}, the formulae describing the relations between two consecutive curved foldlines are presented. In \S \ref{secprop2} we prove that any fold along one foldline can be propagated to any number of rescaled copies of itself, if the scaling factor is small enough. Finally, in \S \ref{infinite} we discuss how the propagation of a fold can turn singular in an arbitrarily abrupt manner, implying that an existence proof of foldability on any pattern with infinitely many prescribed foldlines must involve a control mechanism on the derivative of all orders. The appendix contains a more thorough discussion of the examples in Fig.\ \ref{hypar3povray}, \ref{torus4povray}, employing the formulae from \S \ref{secprop1} to make apparent the regularity of the developables involved.

\section{Space curves and parabolic developables} \label{recall}
If $\gamma$ is an arc-length parametrized $C^3$ curve, denoting the derivative with respect to the arc-length parameter with a prime, we define the Frenet frame of $\gamma$ as the triple of orthonormal vectors $\{T,N,B\}:=\Big\{\gamma',\frac{\gamma''}{\|\gamma''\|},\gamma' \times \frac{\gamma''}{\|\gamma''\|} \Big\}$. At the same time, if the curve is known to be lying on a surface of $\mathbb{R}^3$ whose unit normal at $\gamma(s)$ is $n(s)$, we can define also the Darboux frame as $\{T,u,n\}:=\{\gamma',n \times \gamma' , n\}$. The coefficients that express the first derivative of such bases with respect to the basis itself have significant geometric meanings,
$$
\begin{pmatrix}
T' \\
N' \\
B' 
\end{pmatrix} =\begin{pmatrix}
0 & k & 0 \\
-k & 0  & \tau \\
0 & -\tau  & 0 
\end{pmatrix}
\begin{pmatrix}
T \\
N \\
B 
\end{pmatrix}, \;\;\;\;\;
\begin{pmatrix}
T' \\
u' \\
n' 
\end{pmatrix} =\begin{pmatrix}
0 & k_g & k_n \\
-k_g & 0  & \tau_r \\
-k_n & -\tau_r  & 0 
\end{pmatrix}
\begin{pmatrix}
T \\
u \\
n 
\end{pmatrix}.
$$
The nonnegative function $k$ is called \textit{curvature} of the curve. The \textit{geodesic curvature} $k_g$ with respect to the given surface is the length of the projection of the \textit{curvature vector} $k\cdot N$ to the \textit{tangent plane} of the surface, spanned by $T$ and $u$, and signed with respect to $u$. The \textit{normal curvature} $k_n$ is the signed length of the projection of the same curvature vector to the normal direction $n$. The function $\tau$ is called \textit{torsion} of the curve, while $\tau_r$ is the \textit{relative torsion} with respect to the given surface.

For our purposes it will be useful to express the geometric descriptors above as a function of the angle $\alpha$ between the \textit{osculating plane}, spanned by $T$ and $N$, and the tangent plane of the surface. We measure $\alpha$ anticlockwise by looking at the angle between $B$ and $n$ from the tip of $T$. Then
$$
k_g = \cos (\alpha) k, \;\;\;\;\;\; k_n=-\sin(\alpha)k, \;\;\;\;\;\; \tau_r=\tau+\alpha'. 
$$
For a more extensive treatment and additional insights about the quantities and formulae above the reader may refer to \cite[\S 1-5 and exercise 19 in \S 3-2]{docarmo}.

If a regular surface is of class $C^2$, we can compute at each point and in each tangent direction $v$ its \textit{normal curvature}, that is the curvature of the section obtained by intersecting the surface with the plane spanned by the tangent $v$ and the normal to the surface $n$. Varying $v$, we call \textit{principal curvatures} the maximum $k_1$ and the minimum $k_2$ among the normal curvatures. The product $K:=k_1\cdot k_2$ is the \textit{Gaussian curvature} of the point. The Theorema Egregium by Gauss guarantees that its value is preserved under $C^2$ isometries (\cite[pp.\ 759--760]{onthefund} or \cite[\S 4-3]{docarmo} assuming $C^3$ regularity). A surface that locally can be obtained as image of a planar domain by a $C^h$ isometry is called a $C^h$ \textit{developable}. If $h\geq2$, because of the invariance just discussed, its Gaussian curvature must be everywhere $0$; we call a point of such a surface \textit{parabolic} if the two principal curvatures satisfy (up to relabelling) $k_1 \neq k_2 = 0$ and \textit{flat} if instead $k_1=k_2=0$. 
 
If parabolic points are dense on the surface, it can be shown that a unique straight line (a \textit{ruling}) passes through any of its points and that the tangent plane of the surface along this line is constant. In the rest of the paper we will be interested in developable surfaces that have only parabolic points; in this case its representation as a family of rulings (\textit{ruled parametrization}) enjoys useful regularity properties. 

\begin{theorem}{\cite[pp. 769-770]{onthefund}} \label{thmruling}
Let $S$ be a $C^h$ developable surface with $h \geq 2$ and no flat points. For any point $p\in S$ there exist a $C^h$ arc-length parametrized curve $\gamma: [-\varepsilon,\varepsilon] \rightarrow \mathbb{R}^3$ and a $C^{h-1}$ function $r: [-\varepsilon,\varepsilon] \rightarrow \mathbb{R}^3$, with $\|r\|\equiv 1$, such that, in a neighbourhood of $p$, $S$ can be parametrized as $a(u,t)=\gamma(u)+t\cdot r(u)$. Moreover, fixed $\bar{u}$, the tangent plane of the surface is constant along the ruling $a(\bar{u},t)$.
\end{theorem}

Before moving to the next section we recall the elementary fact at the core of the local geometry of curved folding.

\begin{lemma}{\cite[\S 4-2]{docarmo}} \label{geodesic}
The geodesic curvature of a planar curve is preserved under isometries of the planar domain in which it is contained. 
\end{lemma}

\section{Local curved folding} \label{local} 

In the following, \textit{folding along} a \textit{foldline}, which is a curve contained in an open domain of $\mathbb{R}^2$, means the isometric mapping of such a planar domain onto two $C^1$ \textit{good surfaces} (decomposable as a finite complex of $C^2$ regions joined by vertices and $C^2$ edges, \cite{nonexistence}) that meet with $C^0$ regularity (and not more) along the image of the foldline. With folding the foldline \textit{onto} a space curve we mean folding along the foldline in such a way that its image under the isometry is the given space curve, which we call the \textit{ridge}. A visualization of a curved fold about a point of a $C^2$ ridge is given in Fig.\ \ref{firstfolding}. In \cite{demainelens}, where local curved folding onto good surfaces is thoroughly studied, it is shown that, in order to construct an isometry on both sides of the foldine, the regularity of the ridge cannot be $C^1$ while not being $C^2$ so, unless the ridge is kinked, its Frenet frame is well-defined. In the rest of the paper we will be mainly interested in folding along foldlines and onto ridges whose regularity is at least $C^3$.

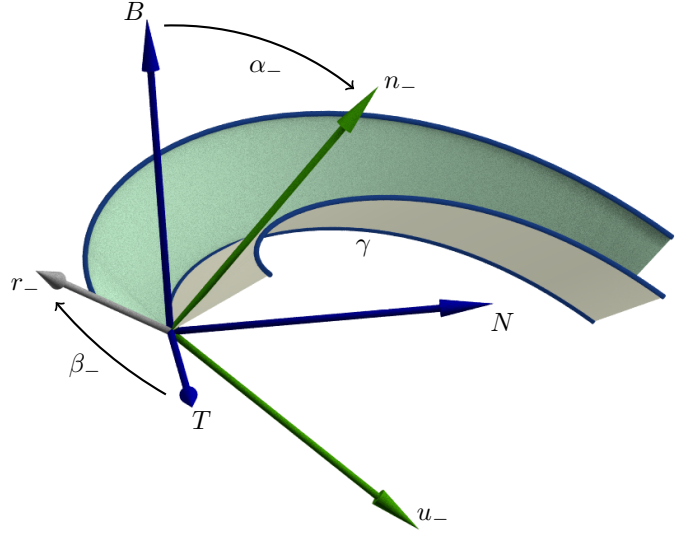
\begin{figure}
\begin{tikzpicture}
       \node[anchor=south west,inner sep=0] at (0,0) {\scalebox{0.7}{\includegraphics[width=\textwidth]{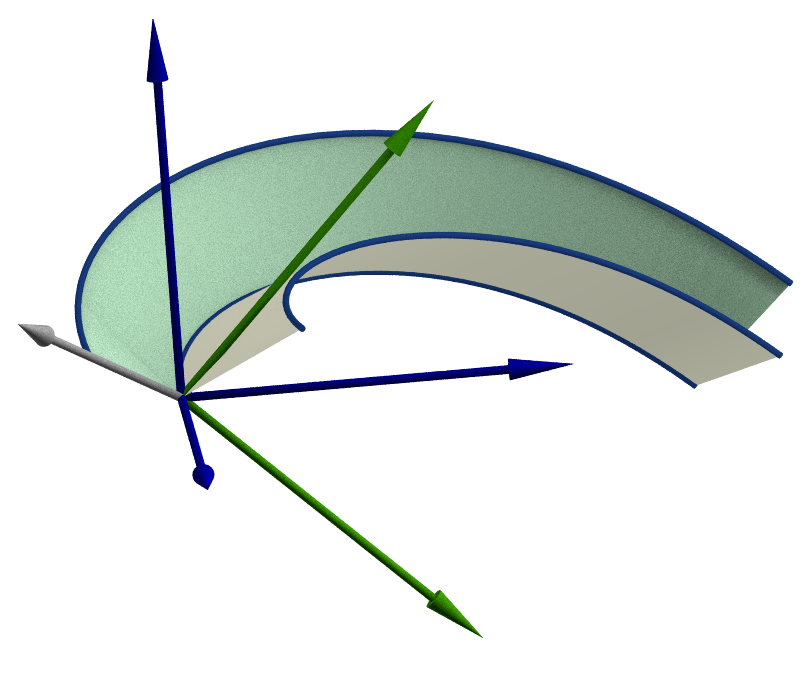}}};
       
       \pgfmathsetmacro{\SCALE}{1} 
    
       \node at ($\SCALE*(2.4,1.85)$) {$T$};
       \node at ($\SCALE*(6.4,3.15)$) {$N$};
       \node at ($\SCALE*(1.5,7.3)$) {$B$};
       \node at ($\SCALE*(3.26,6.52)$) {$\alpha_-$};
       \node at ($\SCALE*(5.05,6.3)$) {$n_-$};
       \node at ($\SCALE*(5.48,0.55)$) {$u_-$};
       \node at ($\SCALE*(0.06,3.62)$) {$r_-$};
        \node at ($\SCALE*(0.85,2.57)$) {$\beta_-$};
       \draw[thick,->] ($\SCALE*(1.9,7.1)$) arc (92:  51:3.8cm);
       \draw[thick,<-] ($\SCALE*(0.46,3.42)$) arc (220:  240:5.5cm);
       
        \node at ($\SCALE*(4.55,4.14)$) {$\gamma$};
\end{tikzpicture}
\caption{$\{T,N,B\}$ is the Frenet frame of the ridge, while $\{T,u_-,n_-\}$ its Darboux frame with respect to the outer green surface oriented by $n_-$. The tangent plane to the surface is spanned by $T$ and $u_-$ and the the ruling direction $r_-$ lies on it. The Darboux frame of the inner white surface can be obtained by simply rotating $\{T,u_-,n_-\}$ by $-2\alpha_-$ about $T$.} \label{firstfolding}
\end{figure}

In order to fix the notation and to explain why Fig.\ \ref{firstfolding} is substantially the only way a local curved fold can look like we recall a couple of formulae from \cite{more}. These describe how paper locally folds along a given $C^3$ foldline with curvature $k_g>0$ once a $C^3$ ridge is prescribed. We call $n_+$ the normal to the developable such that the angle $\alpha_+$ between the binormal vector $B$ of the Frenet frame of $\gamma$ and $n_+$, measured anticlockwise with respect to $T=\gamma'$, has value $0<\alpha_+<\frac{\pi}{2}$. Analogously, the normal $n_-$ and the angle $\alpha_-$ are defined for the other developable to satisfy $0>\alpha_->-\frac{\pi}{2}$.
Since geodesic curvature is preserved by Lemma \ref{geodesic}, denoting with $k$ the curvature of the ridge $\gamma$, we have 
$$
\cos (\alpha_+) k = k_g = \cos (\alpha_-) k, 
$$
and therefore $\alpha_+=-\alpha_-$ if, as in our definition of folding, the transition from one side to the other must be just $C^0$. More precisely, a fold is called \textit{proper} when the above relation is well defined and $\alpha_+\neq 0, \frac{\pi}{2}$ which is the case iff $k>k_g$ and $k_g\neq0$. For intuition, $\alpha_+ \neq 0$ ensures some folding is actually happening and $\alpha_+ \neq  \frac{\pi}{2}$ that the developables on the two sides do not overlap each other. 


\begin{lemma}{\cite{more}} \label{foldridge}
Given a $C^3$ foldline $\bar{\gamma}$ and a $C^3$ ridge of the same length $\gamma$ with curvatures $0<k_g<k$, then, on the two sides of the osculating plane of $\gamma$, two different proper folds are possible along $\bar{\gamma}$ onto $\gamma$ and the unit directions of the rulings of the developables are given by 
$$
r_S=\frac{\tau_{(r,S)} T -k_{(n,S)} \bigl(\cos (\alpha_{S}) N + \sin(\alpha_{S}) B\bigl)}{\sqrt{\tau_{(r,S)}^2+k_{(n,S)}^2}}, \;\; S\in \{+,-\},
$$
where $\{T,N,B\}$ is the Frenet frame of $\gamma$. The symbols $\tau_{(r,+)}$, $k_{(n,+)}$ denote the relative torsion resp. the normal curvature of $\gamma$ with respect to the developable whose normal $n_+$ forms with $B$ the angle $0<\alpha_+<\frac{\pi}{2}$, when measured anticlockwise with respect to $T$. Analogous notation is used for $\tau_{(r,-)}, k_{(n,-)}$, $-\frac{\pi}{2}<\alpha_-<0$ and $n_-$ for the second developable.
\end{lemma}

The next lemma relates the normal curvature and relative torsion of a ridge with respect to one developable to the normal curvature and relative torsion of the same ridge with respect to the developable on the opposite side.

\begin{lemma} \label{fliptors}
In the notation of Lemma \ref{foldridge}, we have the equalities
\begin{equation*}
\begin{split}
k_{(n,S)} & = - k_{(n,\bar{S})}, \\
\tau_{(r,S)} & =\tau_{(r,{\bar{S}})}-2\alpha_{\bar{S}}' =  \tau_{(r,{\bar{S}})}-2\frac{k_g'k_{(n,{\bar{S}})}-k_gk_{(n,{\bar{S}})}'}{k_g^2+k_{(n,{\bar{S}})}^2},
\end{split}
\end{equation*}
where $S \in \{+,-\}$ and $\bar{S}$ is the opposite sign of $S$.
\end{lemma}
\begin{proof}
By direct computation, 
\begin{equation*}
\begin{split}
\tau_{(r,S)} & =\tau+\alpha_S'=\tau_{(r,\bar{S})}-\alpha_{\bar{S}}' -\alpha_{\bar{S}}'=\tau_{(r,{\bar{S}})}-2\alpha_{\bar{S}}' \\
& = \tau_{(r,{\bar{S}})}+2 \frac{(\cos(\alpha_{\bar{S}}))'}{\sin(\alpha_{\bar{S}})} = \tau_{(r,{\bar{S}})}-2 \frac{\bigg(k_g/\Big(\sqrt{k_g^2+k_{(n,\bar{S})}^2}\Big)\bigg)'}{k_{(n,\bar{S})}/\Big(\sqrt{k_g^2+k_{(n,\bar{S})}^2}\Big)}  \\ 
& = \tau_{(r,{\bar{S}})}-2\frac{k_g'k_{(n,{\bar{S}})}-k_gk_{(n,{\bar{S}})}'}{k_g^2+k_{(n,{\bar{S}})}^2}. \qedhere
\end{split}
\end{equation*}
\end{proof}

By calling $\beta_S$ the functions measuring, anticlockwise with respect to $B$, the angle between $T$ and $r_S$, direct computations provide the following lemma.

\begin{lemma} \label{rulingangle}
Given a proper fold along the $C^3$ foldline $\bar{\gamma}$ onto the $C^3$ ridge $\gamma$, the angles $\beta_S$ between $T$ and $r_S$, for $S \in \{+,-\}$, satisfy
$$
\cos(\beta_S)= \frac{\tau_{(r,S)}}{\sqrt{\tau_{(r,S)}^2+k_{(n,S)}^2}}, \;\; \sin(\beta_S)= \frac{-k_{(n,S)}}{\sqrt{\tau_{(r,S)}^2+k_{(n,S)}^2}}, \;\; \cot(\beta_S)=-\frac{\tau_{(r,S)}}{k_{(n,S)}}$$
and
\begin{equation*}
\begin{split}
\beta_S' = -\frac{(\cos(\beta_S))'}{\sin(\beta_S)}= \frac{\tau_{(r,S)}'k_{(n,S)}-k_{(n,S)}'\tau_{(r,S)}}{\tau_{(r,S)}^2+k_{(n,S)}^2}. 
\end{split} 
\end{equation*}
\end{lemma}

\begin{lemma} \label{secondcurv} The developable surfaces on the two sides of a proper curved fold have no planar points.
\end{lemma}

\begin{proof} 

By knowing the normal curvature $k_{(n,S)}$ of the ridge with respect to the developable surface and the angle $\beta_S$ its tangent forms with the ruling direction, we can retrieve the nonzero principal curvature $k_{(p,S)}$ by using Euler's formula \cite[\S 3-2]{docarmo},
$$
k_{(p,S)}=\frac{k_{(n,S)}}{\sin(\beta_S)^2}.
$$
Since we are considering proper folds, this expression is well defined and nonzero. Finally, it is a classical result that if a ruling contains a parabolic, resp. a flat point, then all of its points must be parabolic, resp. flat \cite[Cor. 6, Chap. 5]{spivak}.

\end{proof}


Although in general the regularity of the ruled parametrization of a developable surface is not greater than $C^0$ (see \cite{ushakovasy} for an explicit analysis of this phenomenon), if the developable presents no planar rulings as in the case of a proper fold then the regularity of the surface passes over to the ruled parametrization in the way described in Theorem \ref{thmruling}. In particular, if the foldline and the ridge are of regularity class $C^h$ then the ruled parametrization is $C^{h-2}$.

The last task we tackle in this section is concerned with locating singular points of the developables of a proper fold, that is identifying the so called \textit{regression curve}, obtained as the envelope of the family of rulings of the developable:
$$
R_S=\gamma- \frac{\langle \gamma',r_S' \rangle}{\langle r_S',r_S'\rangle} r_S = \gamma+ \frac{\sin(\beta_S)}{\beta_S'+k_g} r_S.
$$
This expression is easy to obtain by computing the limit intersection of two rulings approaching each other in the developed state (the formula can for example be found in \cite{sabitovlow}, where developables of low smoothness are investigated in relation to their regression curve). In the setting of proper folding, if we assume the developable is $C^2$ then $r$ is $C^1$ and the regression curve must be projectively continuous, that is it can possibly have points at infinity. If $r$ is just piecewise $C^1$, the expression above is still well-defined if one allows jump discontinuities to occur. 

\section{Local folding along closed convex curves} \label{localconvex}
In this short section we provide two ways one can construct closed space curves onto which it is locally possible to fold along closed convex foldlines. The fact that the curvature of the ridge must everywhere be strictly greater then the curvature of the foldline (Lemma \ref{foldridge}) implies a necessary condition to proper fold along a closed curve \cite{more}: the total curvature of the ridge must be strictly greater than $2\pi$, preventing it from lying in a plane by Fenchel's theorem \cite{fenchekru}.

\subsection*{Ridges on a sphere} 
If a curve on the unit sphere is longer than $2\pi$ then, by adequate rescaling, it is possible to fold onto the curve along the unit circle, i.e. it is possible to extend the isometry between the two curves to a local curved fold.

\begin{lemma} \label{main} Let $\omega$ be a closed $C^3$ curve of length $L>2\pi$ on the unit sphere, then it is possible to proper fold along the unit circle onto $\gamma:=\frac{2\pi}{L}\omega$.
\end{lemma}

\begin{proof} 

Since $\omega$ lies on the unit sphere, its curvature is greater or equal to $1$, the normal curvature with respect to the sphere being everywhere $1$. Therefore, the curvature $k$ of $\gamma$ satisfies $k>1$.  This guarantees that the unit circle and the ridge $\gamma$ satisfy the hypotheses of Lemma \ref{foldridge}.

\end{proof}

The two inner developables of Fig.\ \ref{hypar3povray} (in green) are an example of a proper fold along the unit circle obtained by such a construction.

\subsection*{Ridges on a torus} 
Toroidal curves are another interesting class of space curves suitable for proper folding along any convex closed foldline. For $a \in \mathbb{R}$, $p,q \in \mathbb{N}$ and $\lambda:=q/p$, we consider the family of curves on $[0,p 2\pi]$ given by
$$
\omega_{a,(p,q)} (t):=\bigl(  \bigl(a+\cos(\lambda t)\bigl)\cos( t), \bigl(a+\cos(\lambda t)\bigl)\sin(t),\sin(\lambda t)  \bigl).
$$
For any fixed value of $a$, the curvature of the curve can be made arbitrarily close to $1$ everywhere by picking a large value of $q$. Since the length of the curve is monotone in $a$, by rescaling the curve to be of length $1$, we can obtain ridges of arbitrarily large minimum curvature. These are therefore suitable to proper fold along any closed convex foldline. By writing down the expression for the torsion one can additionally observe that in this regime its value tends to be $0$ everywhere. Since the first derivative with respect to $t$ of the curvature function can be made everywhere arbitrarily small and with that the angle $\alpha$ between the osculating and the tangent planes close to constant, we can even force the rulings emanating from the ridge to be about orthogonal to the tangent direction everywhere along the curve. Self-intersections of the developables obtained may occur.

The idea of employing a toroidal curve as the ridge of a curved fold was already mentioned as an example in \cite{masterthesis}. Moreover, in \cite{hprinciple} it is shown that in the isotopy class of any $C^2$ knot of the space there exists a $C^\infty$ knot of constant curvature which is arbitrarily close to the first one both in trajectory and tangent direction. Since the curvature of the approximating knot can be chosen to be any value larger then the maximum curvature of the starting knot, constructions as the one we described for toroidal curves are possible in a much broader setting. 

\begin{figure}[ht!]
\scalebox{0.35}{\includegraphics{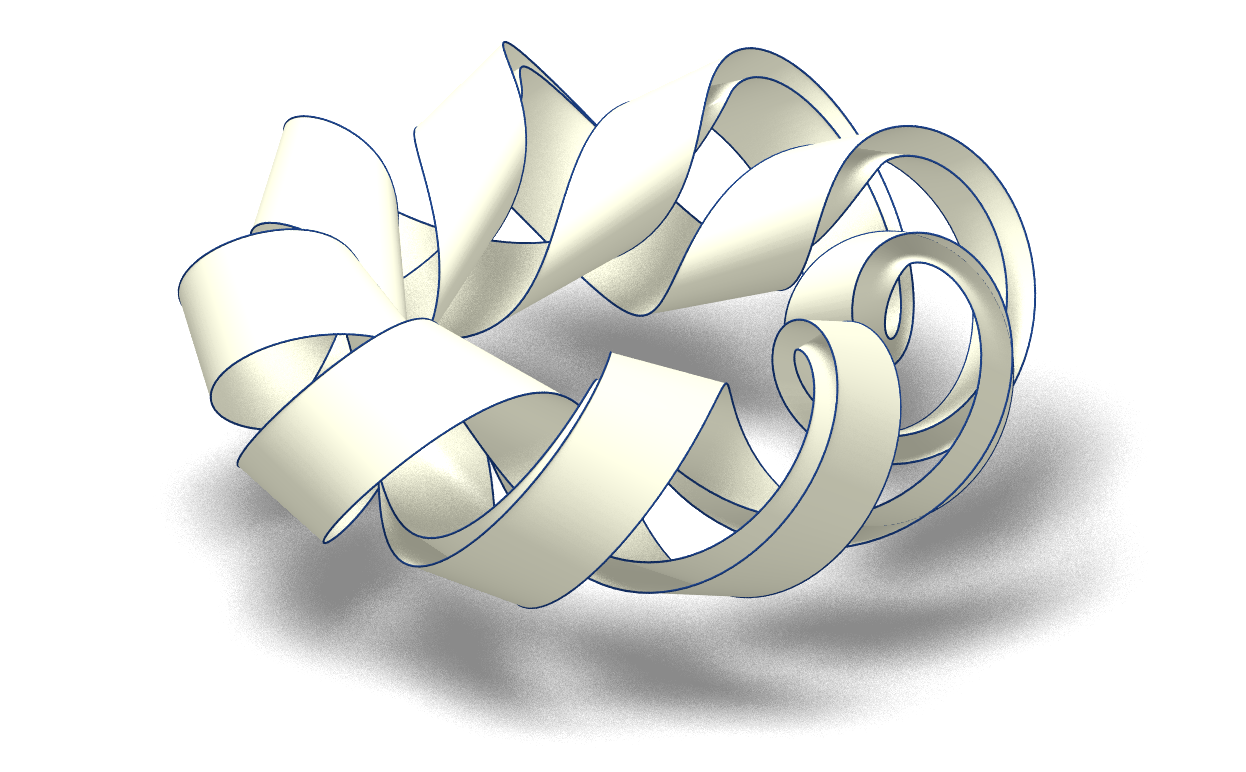}} \caption{Fold along the unit circle of an annulus of width $2/10$ onto the rescaled toroidal knot $\omega_{3,(9,2)}$.}
\end{figure}

\section{Propagation to the next foldline} \label{secprop1} In this section we look at folds involving two foldlines. In particular we first fold along the first foldline by prescribing a ridge as we did in \S \ref{local} and then we induce, if the isometry on one of the two sides extends suitably till the second foldline, a proper fold (consistent with the first one) along such a curve as well. 

Let $\bar{\gamma}_1$ and $\bar{\gamma}_2$ be two non-intersecting planar curves of nonzero curvature in an open domain $D$ of $\mathbb{R}^2$. We construct a proper fold along $\bar{\gamma}_1$ onto the ridge $\gamma_1$ and assume that the isometries on the two sides of it extend to the whole domain $D $; this means that exactly one ruling passes through any of the point of $D$ different from those of $\bar{\gamma}_1$ and that the preimage of the regression curve of the developable is not reached within the domain along the ruling. We also assume that no rulings coming out of $\bar{\gamma}_1$ are tangent to $\bar{\gamma}_2$ where they intersect it for the first time. Under these premises, the restriction of the isometry of $D$ to the second foldine induces a ridge $\gamma_2$ onto which it is possible to proper fold along $\bar{\gamma}_2$. We guarantee this by arguing that all points of the developables of a proper fold are parabolic by Lemma \ref{secondcurv} and that no rulings are tangent to $\gamma_2$; therefore, its normal curvature with respect to the developable must be different from zero, ensuring that also $\bar{\gamma}_2$ and $\gamma_2$ satisfy the hypotheses of Lemma \ref{foldridge}. 

Note that our assumptions do not imply that all the rulings emanating from $\gamma_1$ interesect $\gamma_2$. Nevertheless, by possibly restricting $D$ we can force a bijection between the two curves mediated by the family of rulings, associating to a point $\gamma_1(s_1)$ the point $\gamma_2\bigl(s_2(s_1)\bigl)$, obtained as the first intersection of the ruling emanating from $\gamma_1(s_1)$ with $\gamma_2$, $s_1$ and $s_2$ being the respective arc-legth parameters (Fig.\ \ref{propagationtonext}). 

If we want now to \textit{propagate} the proper fold along $\bar{\gamma}_2$ onto $\gamma_2$ and we want it to be consistent with the first one, we have only one choice, since by Lemma \ref{foldridge} the developables of a proper fold lie on the same side of the osculating plane of the ridge and one of them is already prescribed by the first fold. 

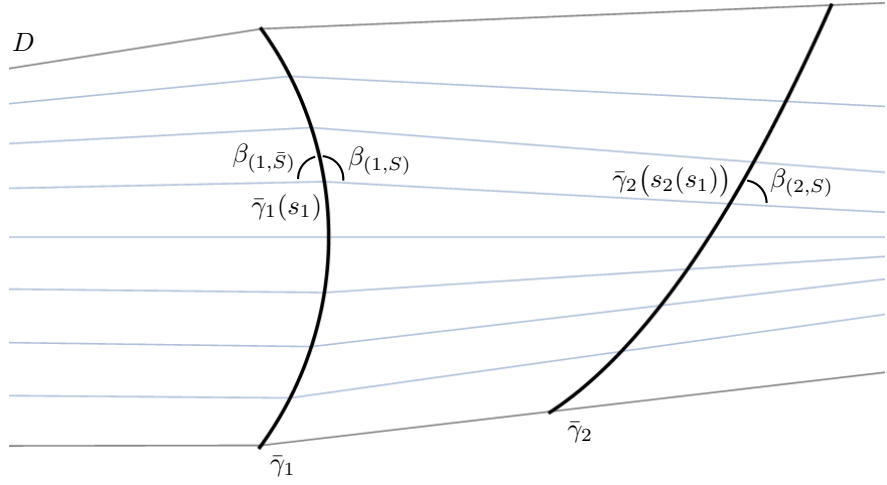
\begin{figure}
\begin{tikzpicture}
          \node[anchor=south west,inner sep=0] at (0,0) {\scalebox{0.92}{\includegraphics[width=\textwidth]{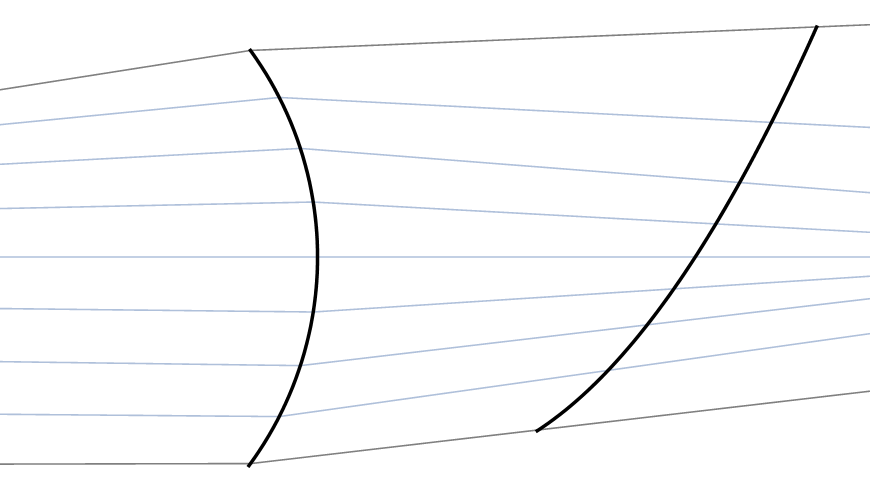}}};
           \draw[thick] (4.45,3.85) arc (-5:  85:0.3cm);
           \draw[thick] (3.85,3.87) arc (181:  104:0.3cm);
           \draw[thick] (9.82,3.85) arc (82:  0:0.3cm);
           \node at (0.2,5.7) {$D$};
           \node at (3.7,3.5) {$\bar{\gamma}_1(s_1)$};
           \node at (4.94,4.12) {$\beta_{(1,S)}$};
           \node at (3.4,4.15) {$\beta_{(1,\bar{S})}$};
           \node at (10.55,3.82) {$\beta_{(2,S)}$};
            \node at (8.82,3.85) {$\bar{\gamma}_2\bigl(s_2(s_1)\bigl)$};
            
            \node at (3.65,0.05) {$\bar{\gamma}_1$};
            \node at (7.6,0.5) {$\bar{\gamma}_2$};
\end{tikzpicture}
\caption{One of the two isometries obteined after proper folding along $\bar{\gamma}_1$ extends to the foldline $\bar{\gamma}_2$. The domain has possibly been trimmed to guarantee a bijection provided by the rulings between the foldlines.} \label{propagationtonext}
\end{figure}

In the next lemma we express the normal curvature and the relative torsion of $\gamma_2$ with respect to the developable obtained by proper folding onto $\gamma_1$, in function of the normal curvature and the relative torsion of $\gamma_1$.


\begin{lemma} \label{propdelta}
Let $\gamma_1$ and $\gamma_2$ be two non-intersecting curves on a $C^2$ developable surface whose points are all parabolic. We assume that the two curves have nonzero geodesic curvature, nonzero normal curvature (their tangents are never parallel to the rulings) and that a bijection $\gamma_1(s_1) \longleftrightarrow \gamma_2\bigl(s_2(s_1)\bigl)$ is induced by considering the first intersection point between $\gamma_2$ and the ruling through $\gamma(s_1)$. If $\delta$ is the angle between the tangent vectors at correspondent points $\gamma_1'(s_1)$ and $\gamma_2'\bigl(s_2(s_1)\bigl)$, measured anticlockwise with respect to the surface normal $n$, then
$$
\begin{pmatrix}
k_{2,n} \\
\tau_{2,r}
\end{pmatrix} = \frac{1}{s_2'}\begin{pmatrix}
\;\;\,\cos(\delta)k_{1,n} +\sin(\delta)\tau_{1,r} \\
-\sin(\delta)k_{1,n} +\cos(\delta)\tau_{1,r}
\end{pmatrix} = \hspace{-1pt} \frac{1}{s_2'} R_{-\delta}
\begin{pmatrix}
k_{1,n} \\
\tau_{1,r}
\end{pmatrix}= \hspace{-1pt} \frac{k_{2,g}}{\delta'+k_{1,g}} R_{-\delta}
\begin{pmatrix}
k_{1,n} \\
\tau_{1,r}
\end{pmatrix},
$$
where $k_{i,g}$,$k_{i,n}$ and $\tau_{i,r}$ respectively are the geodesic curvature, the normal curvature and the relative torsion of $\gamma_i$ with respect to the developable, for $i\in\{1,2\}$. $R_\omega$ denotes the anticlockwise rotation by the angle $\omega$.
\end{lemma}
\begin{proof}
Let $n_1$, $n_2$ be the restriction of the surface normal $n$ to the curves $\gamma_1$, $\gamma_2$ and $\{\gamma_1',u_1,n_1\}$, $\{\gamma_2',u_2,n_2\}$ the respective Darboux frames. Since the surface normal is constant along the ruling we have
\begin{equation*}
\begin{split}
& s_2'(s_1) \Big(-k_{2,n} (s_2(s_1)) \gamma_2'\bigl(s_2(s_1)\bigl) -\tau_{2,r} (s_2(s_1)) u_2\bigl(s_2(s_1)\bigl)\Big) \\ &=s_2'(s_1)n_2'\bigl(s_2(s_1)\bigl) =\bigl(n_2\bigl(s_2(s_1)\bigl)\bigl)'=\bigl(n_1(s_1)\bigl)'= -k_{1,n} (s_1) \gamma_1'(s_1) -\tau_{1,r} (s_1) u_1(s_1).
\end{split}
\end{equation*}
The vectors $\gamma'_2$ and $u_2$ can be obtained rotating respectively $\gamma'_1$ and $u_1$ by $\delta$ about the surface normal $n$ and thus, interpreting $k_{1,n}$, $\tau_{1,r}$ and $k_{2,n}$, $\tau_{2,r}$ as coordinates of the same vector in different bases we get
$$
\begin{pmatrix}
k_{2,n} \\
\tau_{2,r}
\end{pmatrix}= \frac{1}{s_2'} R_{-\delta} \begin{pmatrix}
k_{1,n} \\
\tau_{1,r}
\end{pmatrix}.
$$ 
To express the velocity $s_2'$ we look at the developed state $\bar{\gamma}_1$, $\bar{\gamma}_2$ of the two curves and exploit the relation $\bar{\gamma}_2'=R_\delta \bar{\gamma}_1'$ at correspondent points. By taking the derivative with respect to $s_1$ we obtain
$$
s_2'\bar{\gamma}_2''=\delta' R_\delta R_{\frac{\pi}{2}} \bar{\gamma}_1' + R_\delta \bar{\gamma}_1'',
$$
which provides $s_2'=(\delta'+k_{1,g})/k_{2,g}$.
\end{proof}

We want now to point out an additional way of computing the normal curvature and the relative torsion of the second ridge once a proper fold is prescribed for the first one. This expression will highlight the role played by the regression curve in the propagation and provide a direct formula for computing the velocity of the parametrization of the second ridge induced by the rulings correspondence.

\begin{lemma} \label{ktaureg}
Let $\gamma_1$, $\gamma_2$ be two curves on a developable surface as in Lemma \ref{propdelta}. Let also $\beta_{(1,S)}$ be the angle between the tangent $\gamma_1'(s_1)$ and the ruling direction and $\beta_{(2,S)}$ the one between $\gamma_2'\bigl(s_2(s_1)\bigl)$ and the same ruling direction (Fig.\ \ref{propagationtonext}). If $\bar{v}$ is the distance between $\gamma_1(s_1)$ and $\gamma_2\bigl(s_2(s_1)\bigl)$, then 
\begin{equation*}
\begin{split}
k_{(2,n,S)} = & \;\;\;\; \bigg( \frac{\sin(\beta_{(2,S)})}{\sin(\beta_{(1,S)})} \bigg)^2 \frac{k_{(1,n,S)}}{1-\bar{v}\frac{\beta_{(1,S)}'+k_{(1,g)}}{\sin(\beta_{(1,S)})}}, \\
\tau_{(2,r,S)}  = & -\bigg( \frac{\cos(\beta_{(2,S)})}{\sin(\beta_{(1,S)})} \bigg) \bigg( \frac{\sin(\beta_{(2,S)})}{\sin(\beta_{(1,S)})} \bigg) \frac{k_{(1,n,S)}}{1-\bar{v}\frac{\beta_{(1,S)}'+k_{(1,g)}}{\sin(\beta_{(1,S)})}}. 
\end{split}
\end{equation*}
\end{lemma}

\begin{proof}
As shown in Lemma 5, Chap. 5 of \cite{spivak}, the nonzero principal curvature along a parabolic ruling $\gamma+v\cdot r$ can be written as 
$$
k_{(p,S)}(v)=\frac{1}{\bigl(\sin(\beta_{(1,S)})\bigl)^2}\frac{k_{(1,n,S)}}{1-v\cdot \frac{\beta_{(1,S)}'+k_{(1,g)}}{\sin(\beta_{(1,S)})}}.
$$
This expression has been constructed by requiring the reciprocal of a linear function to attain the value $k_{(1,n,S)} / \bigl(\sin(\beta_{(1,S)})\bigl)^2$ at $v=0$ and being indeterminate at the parameter $v$ corresponding to the intersection with the regression curve. Evaluating at $\bar{v}$, applying Euler's formula and recalling that $\cot(\beta_{(2,S)})=-\tau_{(2,r,S)}/k_{(2,n,S)}$, we obtain the desired formulae for $k_{(2,n,S)}$ and $\tau_{(2,r,S)}$. 
\end{proof}

\begin{lemma} \label{velocity}
Let $\gamma_1$, $\gamma_2$ be two curves on a developable surface as in Lemma \ref{propdelta}. With $\bar{v}$ as in Lemma \ref{ktaureg}, the velocity of the parametrization of the second ridge $\gamma_2\bigl(s_2(s_1)\bigl)$ can be expressed as
$$
s_2'=\frac{\sin(\beta_{(1,S)})}{\sin(\beta_{(2,S)})}\bigg(1-\bar{v}\frac{\beta_{(1,S)}'+k_{(1,g)}}{\sin(\beta_{(1,S)})} \bigg).
$$
\end{lemma}
\begin{proof}
By Lemma \ref{propdelta}, we have
$$
s_2'=\frac{\cos(\delta)k_{(1,n,S)}+\sin(\delta)\tau_{(1,r,S)}}{k_{(2,n,S)}}.
$$
We conclude the claim of the lemma by observing that $\beta_{(2,S)}=\beta_{(1,S)}-\delta$ and hence
$$
\cos(\delta)k_{(1,n,S)}+\sin(\delta)\tau_{(1,r,S)} \hspace{-1.5pt} = \hspace{-1.5pt} k_{(1,n,S)}(\cos(\delta)-\sin(\delta)\cot(\beta_{(1,S)}) \hspace{-1.5pt} =  \hspace{-1.5pt} k_{(1,n,S)}\frac{\sin(\beta_{(2,S)})}{\sin(\beta_{(1,S)})}.
$$
\end{proof}


For what concerns the propagation of a curved fold, we can use Lemma \ref{fliptors} and Lemma \ref{propdelta} to compute, as a function of the normal curvature and relative torsion of the ridge $\gamma_1$, the normal curvature and relative torsion of the second ridge $\gamma_2$ with respect to the developable obtained after proper folding also on the other side of $\gamma_2$. Although the formulae are not simple, such a construction can possibly be iterated to further propagate the fold when several foldlines are prescribed, the propagation being uniquely determined by the chosen foldlines together with the normal curvature and relative torsion of the first ridge. Fig.\ \ref{hypar3povray} and \ref{torus4povray} show two examples of a pleated annulus with multiple folds, drawn via their explicit parametrizations, which have been obtained by the propagation process just described. The details on how to guarantee the regularity of such a construction are given in the appendix.

\begin{prep} \label{formulona}
Let $\gamma_1$, $\gamma_2$ be two curves on a developable surface $M$ as in Lemma \ref{propdelta}, then there is a unique way to propagate the fold onto $\gamma_2$. In more detail, there is a unique way to properly fold onto $\gamma_2$ (along the foldline with the correspective geodesic curvature) in a consistent way with the pre-existing developable $M$. The normal curvature and the relative torsion $k_{(2,n,S)}$, $\tau_{(2,r,S)}$ of $\gamma_2$ with respect to the new developable can be expressed as  



\begin{equation*}
\begin{split}
k_{(2,n,S)} = & - k_{(2,n,\bar{S})}, \\
\tau_{(2,r,S)}  = & \tau_{(2,r,\bar{S})} -\frac{2}{(s_2')^2 \Big(k_{(2,n,\bar{S})}^2+k_{(2,g)}^2 \Big)}  \cdot \bigg( s_2'' k_{(2,n,\bar{S})} k_{(2,g)} \\ & + s_2' \Big(k_{(2,n,\bar{S})} k_{(2,g)}' - \tau_{(2,r,\bar{S})}k_{(2,g)}\big( s_2'k_{(2,g)}-k_{(1,g)} \big) \Big) \\ &  -k_{(2,g)}\Big(\cos\bigl(\beta_{(1,S)}-\beta_{(2,S)}\bigl) k_{(1,n,\bar{S})}'+\sin\bigl(\beta_{(1,S)}-\beta_{(2,S)}\bigl) \tau_{(1,r,\bar{S})}'\Big)\bigg),
\end{split}
\end{equation*}

where $k_{(1,g)}$, $k_{(1,n,\bar{S})}$, $\tau_{(1,r,\bar{S})}$ and  $k_{(2,g)}$, $k_{(2,n,\bar{S})}$, $\tau_{(2,r,\bar{S})}$ respectively are the geodesic curvature, the normal curvature and the relative torsion of $\gamma_1$ and $\gamma_2$ with respect to $M$. Finally, $\beta_{(1,S)}$, $\beta_{(2,S)}$ are the angles between the tangents $\gamma_1'$, $\gamma_2'$ and the ruling direction $r_S$ at corresponding points $\gamma_1(s_1)$, $\gamma_2\bigl(s_2(s_1)\bigl)$.  
\end{prep}
\begin{proof}
Direct computation by Lemma \ref{fliptors} and Lemma \ref{propdelta}.
\end{proof}

\begin{cor}
We can employ $\delta=\beta_{(1,S)}-\beta_{(2,S)}$ to rewrite the formula for the relative torsion from Proposition \ref{formulona} in a slightly more compact way,
\begin{equation*}
\begin{split}
\tau_{(2,r,S)}  = & \tau_{(2,r,\bar{S})} -\frac{2}{k_{(2,n,\bar{S})}^2+k_{(2,g)}^2} \bigg( \frac{k_{(2,g)}}{\delta'+k_{(1,g)}} \bigg)^2 \cdot \bigg(k_{(2,n,\bar{S})}\Big(\delta''+k_{(1,g)}'\Big) \\ & -\tau_{(2,r,\bar{S})}\Big(\delta'+k_{(1,g)}\Big)\delta'-k_{(2,g)}\Big(\cos(\delta) k_{(1,n,\bar{S})}'+\sin(\delta) \tau_{(1,r,\bar{S})}'\Big)\bigg).
\end{split}
\end{equation*}
\end{cor}

\begin{remark}
By Lemma \ref{propdelta} and Lemma \ref{velocity}, $k_{(2,n,\bar{S})}$, $\tau_{(2,r,\bar{S})}$ and $s_2'$ depend only on the prescribed foldlines and on the values of $k_{(1,n,\bar{S})}$ and $\tau_{(1,r,\bar{S})}$ and their first derivative at the point of interest. By this, $s_2''$ depends on the derivatives of $k_{(1,n,\bar{S})}$ and $\tau_{(1,r,\bar{S})}$ up to the second order.
\end{remark}

\begin{obs} \label{appro}
If the two foldlines $\bar{\gamma}_1$ and $\bar{\gamma}_2$ are very close to each other, for example $\bar{\gamma}_2$ being a very gentle offset of $\bar{\gamma}_1$ in the direction of $\bar{\gamma}_1''$, we can approximate $k_{(1,g)} \sim k_{(2,g)}$, $k_{(2,n,\bar{S})} \sim k_{(1,n,\bar{S})}$, $\tau_{(2,r,\bar{S})} \sim \tau_{(1,r,\bar{S})}$ and $\delta \sim 0$ to obtain
$$
\tau_{(2,r,S)}  \sim \tau_{(1,r,\bar{S})} -2\frac{k_{(1,n,\bar{S})}k_{(1,g)}'-k_{(1,n,\bar{S})}'k_{(1,g)}}{\Big(k_{(1,n,\bar{S})}^2+k_{(1,g)}^2 \Big)} \sim \tau_{(1,r,S)}.
$$
This, matching the expression for the relative torsion from Lemma \ref{fliptors}, shows that in this extreme setting the developable on the other side of $M$ with respect to $\gamma_2$ (the new one we want to define) is approximately a continuation of the developable on the other side of $M$ with respect to $\gamma_1$. A more formal discussion of this behaviour will be given in \S \ref{secprop2}.   
\end{obs}

\section{Propagation to several foldlines} \label{secprop2}
In this section we discuss how for any natural number $N$, by choosing a family of uniformly rescaled foldlines close enough to each other, it is possible to propagate the proper fold along the first foldline onto an arbitrary ridge to the remaining $N-1$ foldlines in the way we described in \S \ref{secprop1}. Fig.\ \ref{family} provides a visualization of what we mean with \textit{uniformly rescaled}; the rigorous definition of such a family is given directly in Theorem \ref{thinenough}.

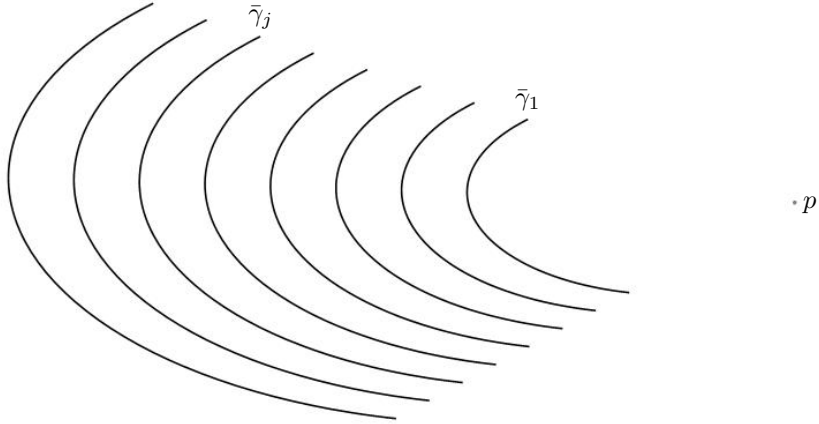
\begin{figure}
\begin{tikzpicture}
          \node[anchor=south west,inner sep=0] at (0,0) {\scalebox{0.92}{\includegraphics[width=\textwidth]{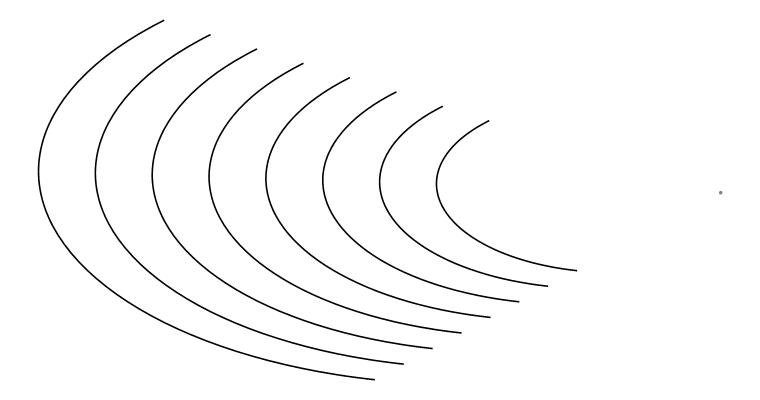}}};
    \node at (7.5,4.55) {$\bar{\gamma}_1$};
    \node at (3.95,5.65) {$\bar{\gamma}_j$};
    \node at (11.25,3.15) {$p$};
\end{tikzpicture}
\caption{Family of foldlines obtained by rescaling $\bar{\gamma}_1$ with respect to the center $p$.} \label{family}
\end{figure}

\begin{theorem} \label{thinenough}
Let $\bar{\gamma}_1$ be a $C^\infty$ foldine with nonzero curvature along which a proper fold onto the $C^\infty$ ridge $\gamma_1$ is locally well defined. Let $p\in\mathbb{R}^2$ be such that no ray $\bar{\gamma}_1-p$ is parallel to $\bar{\gamma}_1'$ then, for any $N \in \mathbb{N}$, there exists a scaling factor $\bar{c}>0$ such that for all $0<c<\bar{c}$ the proper fold along $\bar{\gamma}_1$ propagates to the family of foldlines $\bar{\gamma}_j=(1+(j-1)\cdot c)(\bar{\gamma}_1-p)+p$ for $1<j\leq N$, possibly restricting the definition domain $[a_j,b_j]$ of $\bar{\gamma}_j$ to $[a_j+\rho_{j}(c),b_j-\rho_{j}(c)]$ with $\lim_{c\to 0}\rho_{j}(c)=0$.
\end{theorem}

For the sake of clarity, we proceed by presenting a technical lemma before providing the actual proof of the theorem, which is essentially obtained as a consquence of Observation \ref{appro} plus some work to make the argument rigorous. Given a scaling factor $c$, in the following we will always assume the family $\bar{\gamma}_j$ defined as in the statement of the theorem. If $s_1$ is the arc-length parameter of $\bar{\gamma}_1$ then the arc-length parameter $s_j$ of $\bar{\gamma}_j$ can be expressed as $l_j(s_1)=s_1/(1+(j-1)\cdot c)$.

Although the statements of Theorem \ref{thinenough} and Lemma \ref{technical} are given for a positive value of $c$, this is just for convenience, and analogous conclusions hold for rescaled copies of $\bar{\gamma}_1$ on the same side of $p$ ($c<0$).

\begin{lemma} \label{technical}
Let $\bar{\gamma}_1$ be a $C^\infty$ planar curve with nonzero curvature parametrized by arc-length over the interval $I$. Assume that $p\in\mathbb{R}^2$ is such that no ray $\bar{\gamma}_1-p$ is parallel to $\bar{\gamma}_1'$. If $\bar{r}$ is a $C^\infty$ family of ruling directions defined on $I$, such that no direction $\bar{r}$ is parallel to $\bar{\gamma}_1'$, then for any open interval $A \subset I$ and any $j \in \mathbb{N}$ there exists $\bar{c}>0$ such that for any $0<c<\bar{c}$ the family of rulings direction $\bar{r}$ identifies a bijection between $\bar{\gamma}_j\bigl(l_j(A)\bigl)$ and $\bar{\gamma}_{j+1}\bigl(s_{j+1}\bigl(l_j(A)\bigl)\bigl)$, constructed by considering the intersection of the line $\bar{\gamma}_j\bigl(l_j(s_1)\bigl) + v\cdot \bar{r}(s_1)$ with the curve $\bar{\gamma}_{j+1}$ on the side pointed by $\bar{\gamma}_j-p$. 

Moreover, if $\beta_j$ and $\beta_{j+1}$ are the angles between $\bar{r}$ and $\bar{\gamma}_j'$ and $\bar{\gamma}_{j+1}'$ respectively, then 
\begin{equation*}
\lim_{c \to 0} \abs{\beta_{j+1}^{(h)}\bigl(s_{j+1}(s_j)\bigl)-\beta_{j}^{(h)}(s_j)} = 0, \;\;\;\;\; \forall \, 0 \leq h \in\mathbb{N}, \;\; \text{and}
\end{equation*}
\begin{equation*}
\begin{split}
 \lim_{c\to 0} \abs{s_{j+1}'-1} = \lim_{c\to 0} \abs{s_{j+1}^{(h)}} = 0, \;\;\;\;\; \forall \, 2\leq h \in \mathbb{N}, 
\end{split}
\end{equation*}
Here all derivatives are taken with respect to $s_j$ and $\bar{r}$ is considered as the function $\bar{r}\bigl(l_j^{-1}(s_j)\bigl)$. 
\end{lemma}

\begin{proof}
By the definition of $\bar{\gamma}_j$, since by continuity $\lim_{c \to 0} \bar{\gamma}_{j+1}'\bigl(l_{j+1}(s_1)\bigl)=\lim_{c \to 0} \bar{\gamma}_{j}'\bigl(l_j(s_1)\bigl)=\bar{\gamma}'(s_1)$ we have
$$
\lim_{c \to 0} \abs{\beta_{j+1}\bigl(s_{j+1}(s_j)\bigl)-\beta_{j}(s_j)} = 0.
$$
By the expression for the velocity given in Lemma \ref{velocity} we obtain, for $v(c,s_1)$ such that $\bar{\gamma}_j\bigl(l_j(s_1)\bigl) + v(c,s_1)\cdot \bar{r}(s_1)=\bar{\gamma}_{j+1}\bigl(s_{j+1}\bigl(l_j(s_1)\bigl)\bigl)$,
$$
\lim_{c\to 0} s_{j+1}'= \lim_{c\to 0} \frac{\sin\beta_j}{\sin\beta_{j+1}}\bigg(1-v(c,s_1)\frac{\beta_{j}'+k_{j}}{\sin(\beta_{j})} \bigg)=1,
$$
where $k_{j}$ is the curvature of the $j$-th foldline. For the derivatives of higher order of $s_{j+1},\beta_{j+1}$ and $\beta_j$ we recall from Lemma \ref{propdelta} that $s_{j+1}'=\bigl(k_j+(\beta_j-\beta_{j+1})\bigl)/k_{j+1}$ and therefore
$$
\lim_{c \to 0} \beta_{j+1}'= \lim_{c \to 0} \beta_{j}'-s_{j+1}'k_{j+1}+k_{j} = \lim_{c \to 0} \beta_{j}'.
$$
The statement follows by alternately taking the derivatives with respect to $s_j$ of the expressions for $s_{j+1}'$ and $\beta_{j+1}'$ (induction on the derivatives of lower order).
\end{proof}

We are now ready to prove Theorem \ref{thinenough}.

\begin{proof}[Proof of Theorem \ref{thinenough}]
We proceed by induction, assuming the existence of $\bar{c}$ such that the proper fold along $\bar{\gamma}_1$ onto $\gamma_1$ propagates to the first $N$ foldlines for all $0<c<\bar{c}$. More precisely, the inductive hypothesis we want to iterate claims that for $0<c<\bar{c}$ a bijection between $\bar{\gamma}_j(s_j)$ and $\bar{\gamma}_{j+1}\bigl(s_{j+1}(s_j)\bigl)$ (or equivalently between ${\gamma}_j$ and ${\gamma}_{j+1}$) is identified by the rulings of the developable between the curves for $j<N$, possibly restricting to suitable open sets. Besides, fixed any $\varepsilon>0$ and $H\in\mathbb{N}$ we also ask that in the same range of $c$ the normal curvature and the relative torsion of the ridge $\gamma_j$ with respect to the developable on the opposite side of the scaling center $p$ satisfy 
\begin{equation*}
\begin{split}
\abs{k_{(j,n,S_{j})}^{(h)}\bigl(s_j(s_{j-1}(...s_2(s_1)))\bigl)-k_{(1,n,S_{j})}^{(h)}(s_1)}&<\varepsilon, \;\;\;\;\;   \\
\abs{\,\tau_{(j,r,S_{j})}^{(h)}\bigl(s_j(s_{j-1}(...s_2(s_1)))\bigl)\,-\,\tau_{(1,r,S_{j})}^{(h)}(s_1)}&<\varepsilon, \;\;\;\;\;  \forall \, j\leq N,h\leq H, \\
\end{split}
\end{equation*}
where $S_{j}$ has been assumed without loss of generality being $+$ or $-$ if $j$ is respectively odd or even. The derivatives are taken with respect to the arc-length parameter $s_j$ of the ridge of interest. Finally, still in the inductive hypothesis we ask for guarantees that the speed of the reparametrization mediated by the ruling $s_{j+1}(s_j)$ does not deviate too much from the arc-length, requiring also
\begin{equation*}
\begin{split}
\abs{s_{j+1}'-1}&<\varepsilon, \\
\abs{s_{j+1}^{(h)}}&<\varepsilon,  \;\;\;\;\;  \forall \, j< N,1<h\leq H. \\
\end{split}
\end{equation*}

Assuming the inductive hypothesis for $N\geq 1$, whose basis is simply provided by the knowledge that we can properly fold along $\bar{\gamma}_1$ onto $\gamma_1$, we show that it also holds for $N+1$. We carry over the proof for $N$ odd, the even case being analogous. Since the angle between a ruling and the tangent to the ridge is a continuous function of $k_{(N,n,+)}$ and $\tau_{(N,r,+)}$ and their derivatives, for any $\varepsilon>0$ and $H\in\mathbb{N}$, we can choose $\bar{c}$ small enough to guarantee that
$$
\abs{\beta_{(j,+)}^{(h)}\bigl(s_j(s_{j-1}(...s_2(s_1)))\bigl)-\beta_{(1,+)}^{(h)}(s_1)}<\varepsilon,  \;\;\;\;\;  \forall \, j\leq N,h\leq H.
$$    
The fold along $\bar{\gamma}_1$ being proper, the regression curve of the two developables on which $\gamma_1$ lies is at nonzero distance from the ridge along the ruling. If $\varepsilon$ is small enough, such a property passes over to the regression curves of the developables on the two sides of $\gamma_N$. We can require $\bar{c}$ to be small enough to have $\bar{\gamma}_{N+1}$ within the minimum of such distances and hence obtain the desired bijection $\gamma_{N+1}\bigl( s_{N+1}(s_N) \bigl)$ up to restriction of the definition domains. 

Next step is to control the normal curvature and the relative torsion of $\gamma_{N+1}$ with respect to the developable on the negative side. By making use of the inductive hypothesis on the speed of the bijections $s_j$ and possibly further decreasing $\bar{c}$, we have for $l_j(s_1)=s_1/(1+(j-1)\cdot c)$ 
$$
\abs{\beta_{(j,+)}^{(h)}\bigl(l_j(s_1)\bigl)-\beta_{(1,+)}^{(h)}(s_1)}<\varepsilon,  \;\;\;\;\;  \forall \, j\leq N,h\leq H.
$$
Hence, for $\bar{r}$ chosen to be the development of the rulings direction from $\gamma_1$ to $\gamma_2$, Lemma \ref{technical} together with the triangular inequality allows us to conclude the desired condition on the speed of the reparametrization $s_{N+1}(s_{N})$ and the following bounds on the angle $\beta_{(N+1,+)}$ and its derivatives 
$$
\abs{\beta_{(N+1,+)}^{(h)}\bigl(s_{N+1}(s_N(...s_2(s_1))\bigl)-\beta_{(1,+)}^{(h)}(s_1)}<\varepsilon.
$$
Finally, Proposition \ref{formulona} provides expressions for $k_{(N+1,n,-)}$ and $\tau_{(N+1,r,-)}$ which continuously depend on the values of $\beta_{(N,+)}$ and $\beta_{(N+1,+)}$ and their derivatives up to degree 2. For $\bar{c}$ small enough we can therefore guarantee that
\begin{equation*}
\begin{split}
\abs{k_{(N+1,n,-)} \hspace{-1pt} \bigl(s_{N+1}(...s_2(s_1)))\bigl)\hspace{-1pt} - \hspace{-1pt}\bigl(-k_{(1,n,+)}(s_1)\bigl)}&\hspace{-2pt} <\hspace{-2pt} \varepsilon, \;\;\;\;\;   \\
\abs{\tau_{(N+1,r,-)} \hspace{-1pt} \bigl(s_{N+1}(...s_2(s_1)))\bigl)\hspace{-1pt} - \hspace{-1pt}\Bigg(\hspace{-2pt} \tau_{(1,r,+)} \hspace{-1pt} - \hspace{-1pt} 2\frac{k_{(1,n,+)}k_{(1,g)}'-k_{(1,n,+)}'k_{(1,g)}}{\Big(k_{(1,n,+)}^2+k_{(1,g)}^2 \Big)}\hspace{-2pt} \Bigg)  (s_1)}&\hspace{-2pt} < \hspace{-2pt} \varepsilon,
\end{split}
\end{equation*}
where $k_{(1,g)}$ is the geodesic curvature of $\bar{\gamma}_1$. Again exploiting the condition on the speed of the bijections $s_j$ up to $j=N+1$, the inductive hypothesis on the derivative of $\beta_{(N+1,+)}$ and possibly making $\bar{c}$ smaller enough, for any $H$ we conclude the desired bounds
\begin{equation*}
\begin{split}
\abs{k_{(N+1,n,-)}^{(h)}\bigl(s_{N+1}(s_{j-1}(...s_2(s_1)))\bigl)-k_{(1,n,-)}^{(h)}(s_1)}&<\varepsilon, \;\;\;\;\;   \\
\abs{\,\tau_{(N+1,r,-)}^{(h)}\bigl(s_{N+1}(s_{j-1}(...s_2(s_1)))\bigl)\,-\,\tau_{(1,r,-)}^{(h)}(s_1)}&<\varepsilon, \;\;\;\;\;  \forall \, h\leq H-2. \\
\end{split}
\end{equation*}
This completes the inductive step and with that the proof of the theorem. 
\end{proof}

\begin{obs}
If $\bar{\gamma}_1$ is a closed convex curve then the scaling center $p$ must be in its interior and no restriction of the definition domains is ever needed, once  a scaling factor $c$ small enough to guarantee the propagation has been found.
\end{obs}

\begin{obs}
If we are not interested in mantaining a constant scaling factor $c$, then it is easy to propagate a proper fold onto an arbitrary ridge to infinitely many additional foldlines. We can in fact just proceed by induction: once the $n$-th proper fold is determined we prescribe the $(n+1)$-th foldline by scaling the previous one by a factor small enough to make the curve contained in the interior of the domain on which the isometry identifying the next developable is well-defined.
\end{obs}

\section{Why the propagation to infinitely many prescribed foldlines is hard} \label{infinite} In this section we show that the propagation of a proper fold can turn singular with an arbitrarily abrupt behaviour. More precisely, we will show that for any proper fold involving $N$ foldlines in the sense of \S \ref{secprop2}, we can construct a proper fold over the first $N-1$ foldlines whose ridges are arbitrarily close to those of the first fold up to the derivative of order 3 but such that a non-singular isometry between the $(N-1)$-th and the $N$-th foldline cannot be consistently constructed. This will provide evidence that in general inductive strategies taking into account only derivatives up to a finite order cannot be employed to guarantee the propagation of a curved fold to a prescribed infinite family of foldlines.

\begin{prep} \label{pathologic}
Let $\bar{\gamma}_j$ be a family of $N$ non-intersecting $C^{2N}$ foldlines such that the proper fold along $\bar{\gamma}_1$ onto the $C^{2N}$ ridge $\gamma_1$ propagates sequentially to the foldlines $2$ to $N$ identifying bijections between ridges $\gamma_j(s_j) \leftrightarrow \gamma_{j+1}(s_{j+1}(s_j))$ induced by the rulings correspondence. Then, there exists a ridge $\tilde{\gamma}_{1}$ such that a proper fold along $\bar{\gamma}_1$ onto $\tilde{\gamma}_{1}$ propagates, possibly up to restriction of the definition domains, to the foldlines $2$ to $N-1$, inducing ridges $\tilde{\gamma}_{j}$, but not to the $N$-th foldline, some of the rulings emanating from $\bar{\gamma}_{N-1}$ on the side of $\bar{\gamma}_{N}$ crossing the regression curve before hitting the last foldline. Moreover, for any $\varepsilon>0$, $\tilde{\gamma}_{1}$ can be chosen such that
\begin{equation*}
\begin{split}
& \max\abs{k_{(j,n,S_j)}-\tilde{k}_{(j,n,S_j)}}<\varepsilon, \\
& \max\abs{\,\tau_{(j,r,S_j)}-\tilde{\tau}_{(j,r,S_j)}\,}<\varepsilon, \;\; \forall \, 1 \leq j \leq N-1.
\end{split}
\end{equation*}
where $k_{(j,n,S_j)}$, $\tau_{(j,r,S_j)}$ and $\tilde{k}_{(j,n,S_j)}$, $\tilde{\tau}_{(j,r,S_j)}$ are the normal curvature and relative torsion respectively of $\gamma_j$ and $\tilde{\gamma}_j$ with respect to the developable emanating from the ridge $j$ on the side of $j+1$. 
\end{prep}

Again, we first provide a technical lemma.

\begin{lemma} \label{formal}
In the hypotheses of Theorem \ref{pathologic}, with the notation $k_1:=k_{(1,n,S_1)}$ and $\tau_1:=\tau_{(1,r,S_1)}$, for $j>1$ we have the equalities
\begin{equation*}
\begin{split}
\tau_{(j,r,S_j)}=\tau_1^{(2j-2)}\cdot & f_{j} \Big(s_1, k_1,k_1',...,k_1^{(2j-3)},\tau_1,\tau_1',...,\tau_1^{(2j-3)} \Big)\\ + & g_{j}\Big(s_1, k_1,k_1',...,k_1^{(2j-2)},\tau_1,\tau_1',...,\tau_1^{(2j-3)} \Big), \;\; \text{and}
\end{split}
\end{equation*}
$$
k_{(j,n,S_{j})}=  q_{j} \Big(s_1, k_1,k_1',...,k_1^{(2j-3)},\tau_1,\tau_1',...,\tau_1^{(2j-3)} \Big).
$$
Here $f_{j}$, $g_{j}$ and $q_{j}$ are $C^\infty$ functions depending only on the family of foldlines, and such that $f_{j}$ never attains the value $0$ if evaluated as above (the argument $s_1$ of $k_1$ and $\tau_1$, and $ s_{j}(s_{j-1}(...s_2(s_1))$ of $\tau_{(j,r,S_{j})}$, $k_{(j,n,S_{j})}$ have been omitted for brevity). 
\end{lemma}
\begin{proof}
We proceed by induction, the basis step being provided by Proposition \ref{formulona}, where $s_2''$ is rewritten making use of the expressions for $s_2'$ from Lemma \ref{velocity} and for $\beta'_{(1,S_1)}$ from Lemma \ref{rulingangle}. We employ a similar strategy to prove the inductive step, applying Proposition \ref{formulona} between ridges $j$ and $j+1$. The characterization for $k_{(j+1,n,S_{j+1})}=-k_{(j+1,n,\bar{S}_{j+1})}$ is easily obtained after observing that by Lemma \ref{propdelta} it depends only on $k_{(j,n,\bar{S}_{j})}$ and $\tau_{(j,r,\bar{S}_{j})}$ and their first derivative, and concluding by the inductive hypothesis. We look then at the term $s_{j+1}''k_{(j+1,n,\bar{S}_{j+1})}k_{(j+1,g)}$ containing the derivative of highest order of the expression for $\tau_{(j+1,r,\bar{S}_j)}$. Further decomposing $s_{j+1}''$, by Lemma \ref{velocity} and Lemma \ref{rulingangle} we end up looking at 
$$
-\tau_{(j,r,S_j)}'' \frac{\bar{v} k_{(j,n,S_j)} }{\tau_{(j,r,S_j)}^2+k_{(j,n,S_j)}^2} k_{(j+1,n,\bar{S}_{j+1})}k_{(j+1,g)},
$$
as the term of highest differential order in $\tau_{(j,r,S_j)}$,  where $\bar{v}$ is the distance function between $\gamma_j$ and $\gamma_{j+1}$ along the ruling. Again we conclude by induction after observing that the factor multiplying $\tau_{(j,r,S_j)}''$ is nonzero.
\end{proof}

\begin{proof}[Proof of Proposition \ref{pathologic}]
Without loss of generality in the proof argument, we assume $S_1=+$ and $N$ even with ${S}_{N-1}=\bar{S}_{N}=+$. 

For any $s_1$ and $M \in \mathbb{R}$, we can locally perturbate $\tau_1^{(2N-3)}$ to $\tilde{\tau}_1^{(2N-3)}$, for example with a very steep bump function, to have $\tilde{\tau}_1^{(2N-3)}(s_1)=M$, but still for any $\rho>0$, taking the antiderivatives of $\tilde{\tau}_1^{(2N-3)}$ with suitable boundary conditions 
$$
\max \abs{\tau_1^{(h)}-\tilde{\tau}_1^{(h)}}<\rho, \;\;\;\;\; h<2N-3,
$$
which is possible because we are constraining finitely many antiderivatives defined on compact domains. We define $\tilde{\gamma}_1$ as the ridge having $\tilde{k}_1:=k_1$ and $\tilde{\tau}_1$ as normal curvature and relative torsion and propagate the fold along $\bar{\gamma}_1$ onto such a ridge. By Lemmma \ref{formal} if $\rho$ is small enough, again by continuity and compactness, the normal curvature and the relative torsion of the new ridges $\tilde{\gamma}_1$ to $\tilde{\gamma}_{N-1}$ are arbitrarily close to those of the original ridges ${\gamma}_1$ to ${\gamma}_{N-1}$, which proves the second part from the claim of the theorem.

It remains to show that we can exploit the perturbation freedom we have on $\tilde{\tau}_1^{(2N-3)}$ to (heavily) modify the behaviour of the regression curve of the developable between the ridges $N-1$ and $N$. We do that by recalling that the distance of such a curve from the ridge along a ruling emanating from $\tilde{\gamma}_{N-1}$ is given by
$$
\frac{\sin(\tilde{\beta}_{(N-1,+)})}{\tilde{\beta}_{(N-1,+)}'+k_{(N-1,g)}} \;\; \text{and} \;\; \tilde{\beta}_{(N-1,+)}'=\frac{\tilde{\tau}_{(N-1,r,+)}'\tilde{k}_{(N-1,n,+)}-\tilde{k}_{(N-1,n,+)}'\tilde{\tau}_{(N-1,r,+)}}{\bigl(\tilde{\tau}_{(N-1,r,+)}^2+\tilde{k}_{(N-1,n,+)}^2\bigl)},
$$
where $\tilde{\beta}_{(N-1,+)}$ is the angle function between the ruling and the tangent $\tilde{\gamma}'_{N-1}$ and $k_{(N-1,g)}$ is the curvature of the respective foldline. By Lemma \ref{formal}, $\tilde{\tau}_{(N-1,r,+)}'$ can be made arbitrarily large/small while keeping all the other functions almost unchanged, and with that, since $\tilde{k}_{(N-1,n,+)}\neq 0$, the same behaviour translates to $\tilde{\beta}_{(N-1,+)}'$, hence forcing the point of the regression curve to be arbitrarily located along the ruling and preventing the isometry to be extended to the final foldline. 
\end{proof}

\section{Future work} \label{future}
The construction of \S \ref{secprop2} is artificial in the measure it forces the existence of finitely many folds by exploiting the local guarantees provided by the properness of the first one. It would be nice to see in future years an existence proof that would work on infinitely many uniformly rescaled foldlines. Proposition \ref{pathologic} makes it clear that such a proof would depend on the development of an inductive tool allowing a suitable control not only on the local propagation but also on the derivatives of arbitrary order of the curves involved.   

\section*{Acknowledgments}
The author acknowledges the support of the Austrian Science Fund (FWF): W1230, ``Doctoral Program Discrete Mathematics'' and of SFB-Transregio 109 ``Discretization in Geometry \& Dynamics'' funded by DFG and FWF (I2978).

The author would also like to thank Johannes Wallner for helpful discussions and the support in the use of POV-Ray. 


\bibliographystyle{IEEEtran}
\bibliography{IEEEabrv,bibliography}

\appendix

\section{Folding the annulus}

This appendix is devoted to an explicit application of the formulae obtained in \S \ref{secprop1} of the paper to the annulus folded along conentric circles. 

\begin{lemma} \label{lemmaappendix}
Let $\bar{\gamma}$ be a circle of radius $R$ and center in the origin traversed counter-clockwise and $r$ a unit vector forming with the tangent of the circle at $\bar{\gamma}(s)$ the angle $-\pi<\beta<+\pi$ (measured counter-clockwise). The signed distance $\bar{v}$ between $\bar{\gamma}(s)$ and the closest intersection point between the line $\bar{\gamma}(s)+v \cdot r$ and the scaled circle $(1+c)\bar{\gamma}$ with $c \in \mathbb{R}$, whenever well defined, obeys the formula  
$$
\bar{v}=R\big(\sin (\beta) - \sign\bigl(\sin (\beta)\bigl)  \sqrt{\sin(\beta)^2+c^2+2c}\big).
$$
Besides, the angle $\delta$ between $\bar{\gamma}'(s)$ and the tangent with the second circle at the intersection point satisfies
\begin{equation*}
\begin{split}
\sin(\delta)& =\frac{\bar{v}\cdot \cos(\beta)}{R \cdot (1+c)}, \\
\cos(\delta) & = \frac{R-\bar{v}\cdot \sin(\beta)}{R\cdot (1+c)}.
\end{split}
\end{equation*}
\end{lemma}

\begin{proof}
The lemma follows from elementary computations.
\end{proof}

Given a sequence of concentric circles of radius $R_j$ and a ridge suitable to fold along the $\bar{j}$-th one, formulae from \S \ref{local}, \S \ref{secprop1} and Lemma \ref{lemmaappendix} can be iterated to compute explicit parametrizations of the developables involved in the propagated curved fold. In the notation of \S \ref{secprop1}, two conditions must be met to guarantee the regularity of such surfaces.
\begin{itemize}
\item Setting $c_{(j,\pm)}=R_{j \pm 1}/R_j - 1$, the function
$$
\bar{v}_{(j,S_j)}=R_j\big(\sin (\beta_{(j,S_j)}) - \sign\bigl(\sin (\beta_{(j,S_j)})\bigl)  \sqrt{\sin(\beta_{(j,S_j)})^2+c_{(j,\pm)}^2+2c_{(j,\pm)}}\big)
$$
must be well defined, meaning the ruling intersects the next foldline. 
\item Let $d_{(j,S_j)}=\sin(\beta_{(j,S_j)})/(\beta_{(j,S_j)}'+k_{(j,g)})$ be the signed distance of the regression curve along the ruling, then either 
$$
\sign \bigl( d_{(j,S_j)} \cdot \bar{v}_{(j,S_j)} \bigl) < 0 \text{ or }
\abs{\bar{v}_{(j,S_j)}} \geq \abs{d_{(j,S_j)}},
$$
meaning the regression curve is not intersected before the ruling reaches the next foldline. 
\end{itemize}

The regularity of the two folds represented in Fig.\ \ref{hypar3povray}, \ref{torus4povray} is guaranteed by comparing the values of $\bar{v}_{(j,S_j)}$ and $d_{(j,S_j)}$ computed with the Mathematica code available at \cite{github} (Fig.\ \ref{regVSvHypar}, \ref{regVSvTorus}). For the intersection between the unit sphere and the hyperbolic paraboloid $z-3xy=0$ we provide the plots for one of the four arcs equivalent up to reflection, which can be parametrized as  
$$
\Bigg( t, \sqrt{\frac{1-t^2}{1+9t^2}}, 3 t \sqrt{\frac{1-t^2}{1+9t^2}} \Bigg) \text{ on } \bigg[-\sqrt{(-1+\sqrt{10})/9},\sqrt{(-1+\sqrt{10})/9} \; \bigg].
$$
For the toroidal curve we use the parametrization from \S \ref{localconvex} and restrict the plot to $[0,2\pi]$, which corresponds to one of the five arcs equivalent up to rotation. In both cases the parametrizations are not arc-length and rather than scaling down the starting ridge to match the length of the unit circle, equivalently to our purposes we have scaled up the unit circle (and accordingly all the concentric foldlines) to match the length of the ridge. It is worth mentioning that in the hyperbolic paraboloid case, although the ruling structure of the outer strip is rather well-behaved, it is not possible to further propagate the folding to an additional strip of the same width, since some of the rulings would cross the regression curve of the induced developable before they could reach the new outer boundary (Fig.\ \ref{badregHypar}).

The plots of the normal curvature and those of the relative torsion of the ridges are also provided (Fig.\ \ref{curvtorsHypar}, \ref{curvtorsTorus}).

\begin{figure} 
\begin{tikzpicture}
          \node[anchor=south west,inner sep=0] at (0,0) {\scalebox{1}{\includegraphics[width=0.9\textwidth]{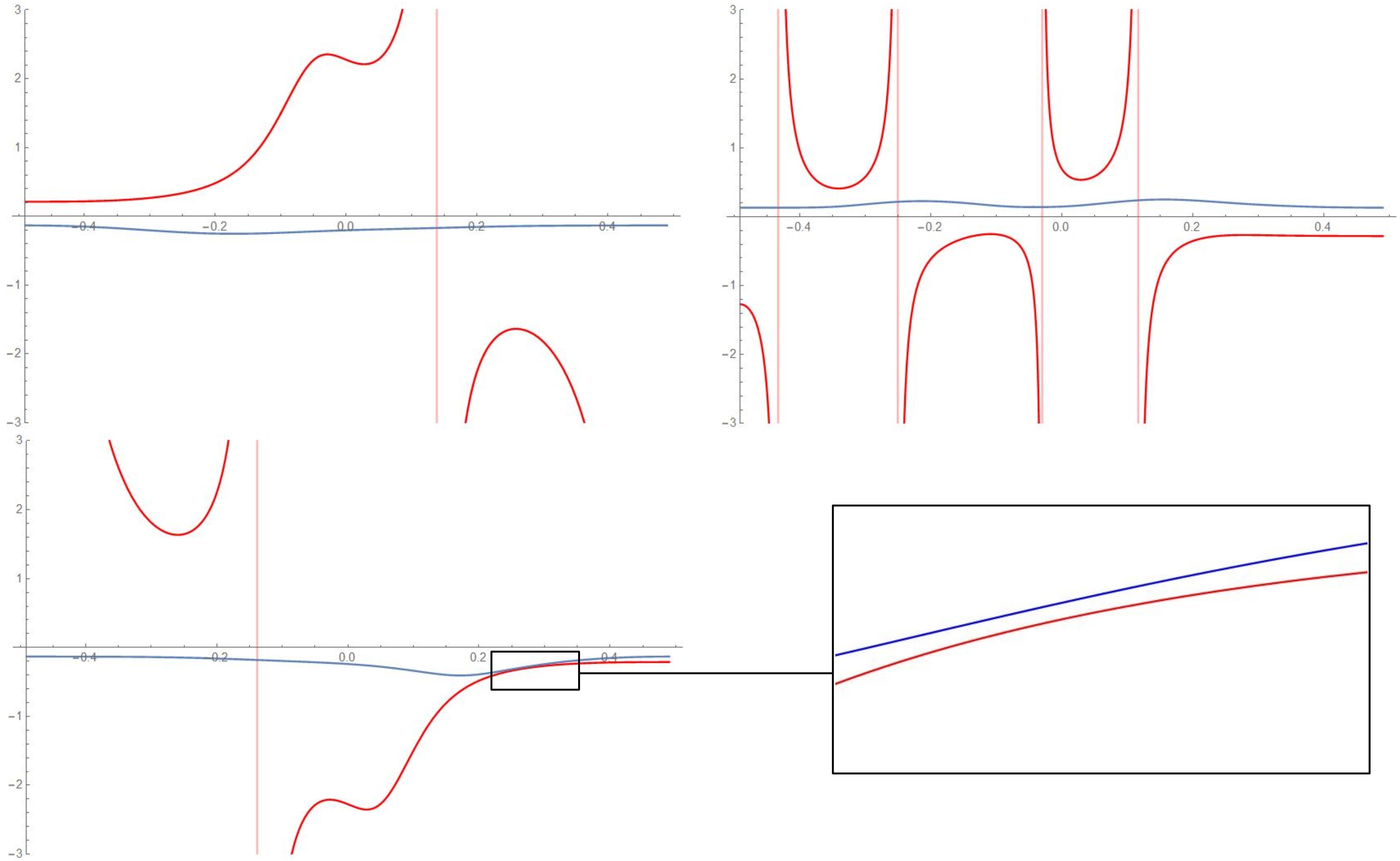}}};
    
    \pgfmathsetmacro{\SCALE}{0.9} 

    \node at ($\SCALE*(2.5,5.4)$) {$\bar{v}_{(1,+)}$};
    \node at ($\SCALE*(1.5,6.4)$) {$d_{(1,+)}$};
    
    \node at ($\SCALE*(11.1,6.2)$) {$\bar{v}_{(2,-)}$};
    \node at ($\SCALE*(11.5,5.3)$) {$d_{(2,-)}$};
    
    \node at ($\SCALE*(3.0,1.5)$) {$\bar{v}_{(1,-)}$};
    \node at ($\SCALE*(1.5,2.7)$) {$d_{(1,-)}$};
    
\end{tikzpicture}
\caption{For the three developables of Fig.\ \ref{torus4povray}, comparison of the signed distance along the ruling between concentric circles $d_{(j,S)}$ and between the ridge and the regression curve $\bar{v}_{(j,S)}$. The surfaces are regular since the rulings reach the next foldline without first crossing the regression curve.} \label{regVSvHypar}
\end{figure}
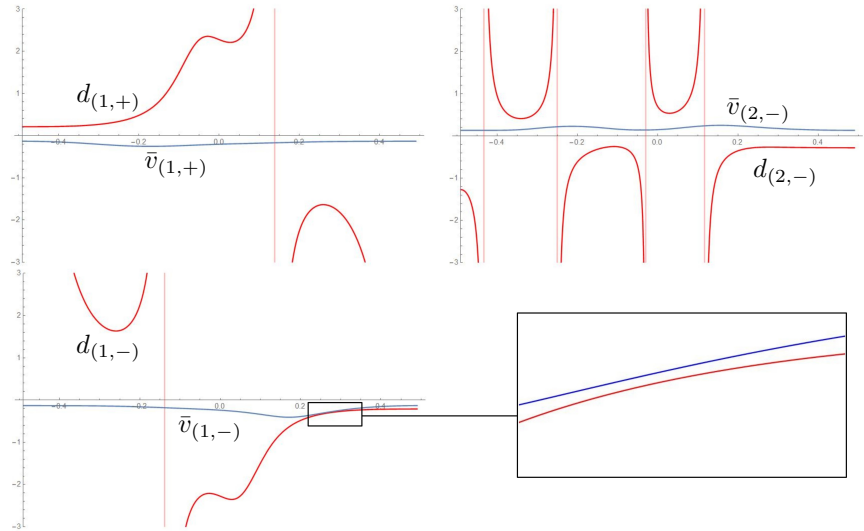

\begin{SCfigure}
\caption{The isometry of a possible third outer strip induced by the propagation would turn singular before reaching the boundary.}\label{badregHypar} 
\scalebox{0.9}{
  \begin{tikzpicture}
          \node[anchor=south west,inner sep=0] at (0,0) {\scalebox{0.53}{\includegraphics[width=\textwidth]{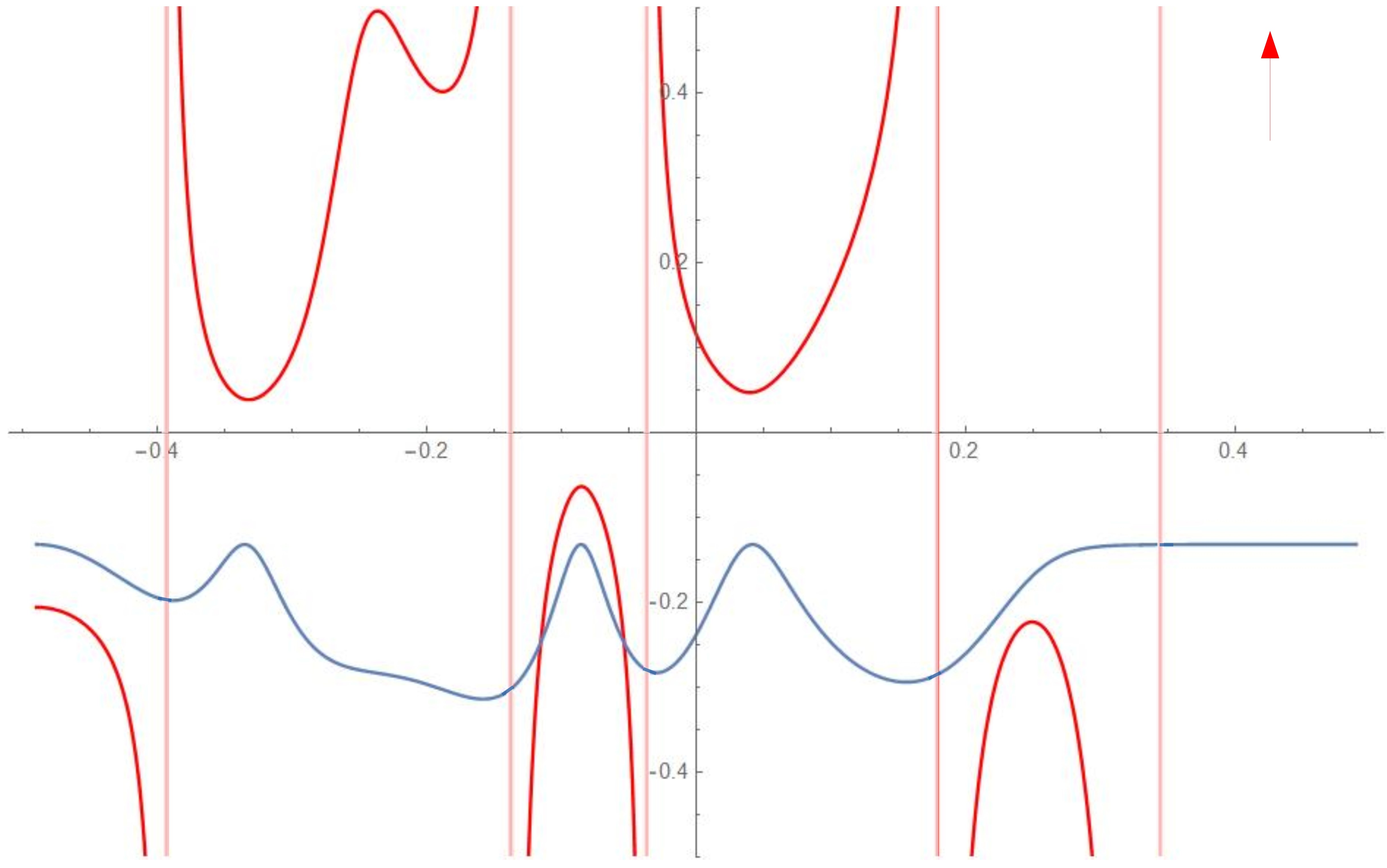}}};
          
      \pgfmathsetmacro{\SCALE}{0.97} 

    \node at ($\SCALE*(6.0,1.9)$) {$\bar{v}_{(3,+)}$};
    \node at ($\SCALE*(2.0,2.6)$) {$d_{(3,+)}$};
    
\end{tikzpicture}}
\end{SCfigure}

\begin{figure} 
\begin{tikzpicture}
          \node[anchor=south west,inner sep=0] at (0,0) {\scalebox{0.9}{\includegraphics[width=\textwidth]{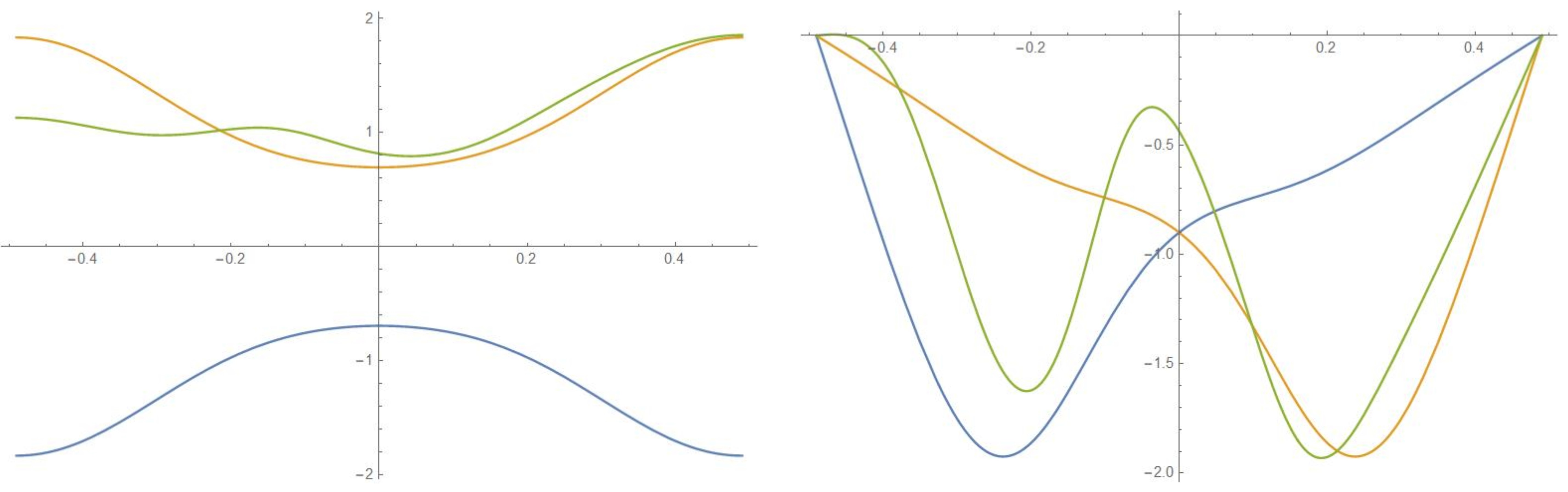}}};
    
    \pgfmathsetmacro{\SCALE}{0.9} 
    
    \node at ($\SCALE*(1.3,0.2)$) {$k_{(1,n,+)}$};
    \node at ($\SCALE*(1.3,3.7)$) {$k_{(1,n,-)}$};
    \node at ($\SCALE*(0.6,2.7)$) {$k_{(2,n,-)}$};
    
     \node at ($\SCALE*(7.2,0.4)$) {$\tau_{(1,r,+)}$};
    \node at ($\SCALE*(10.2,3)$) {$\tau_{(2,r,-)}$};
    \node at ($\SCALE*(8.2,3.1)$) {$\tau_{(1,r,-)}$};
    
\end{tikzpicture}
\caption{In the notation from \S \ref{secprop1}, normal curvature and relative torsion of the two ridges in the fold of Fig.\ \ref{hypar3povray}. Note that $k_{(1,n,+)}$, $\tau_{(1,r,+)}$  and $k_{(1,n,-)}$, $\tau_{(1,r,-)}$ are  respectively the normal curvature and relative torsion of the first ridge (the one from which the propagation starts) w.r.t. the developables on its two sides.} \label{curvtorsHypar}
\end{figure}

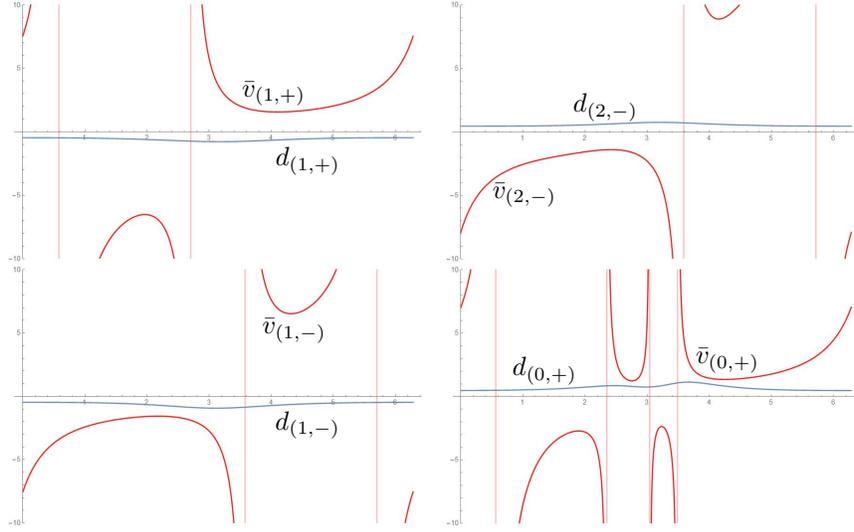
\begin{figure}
\begin{tikzpicture}
          \node[anchor=south west,inner sep=0] at (0,0) {\scalebox{0.9}{\includegraphics[width=\textwidth]{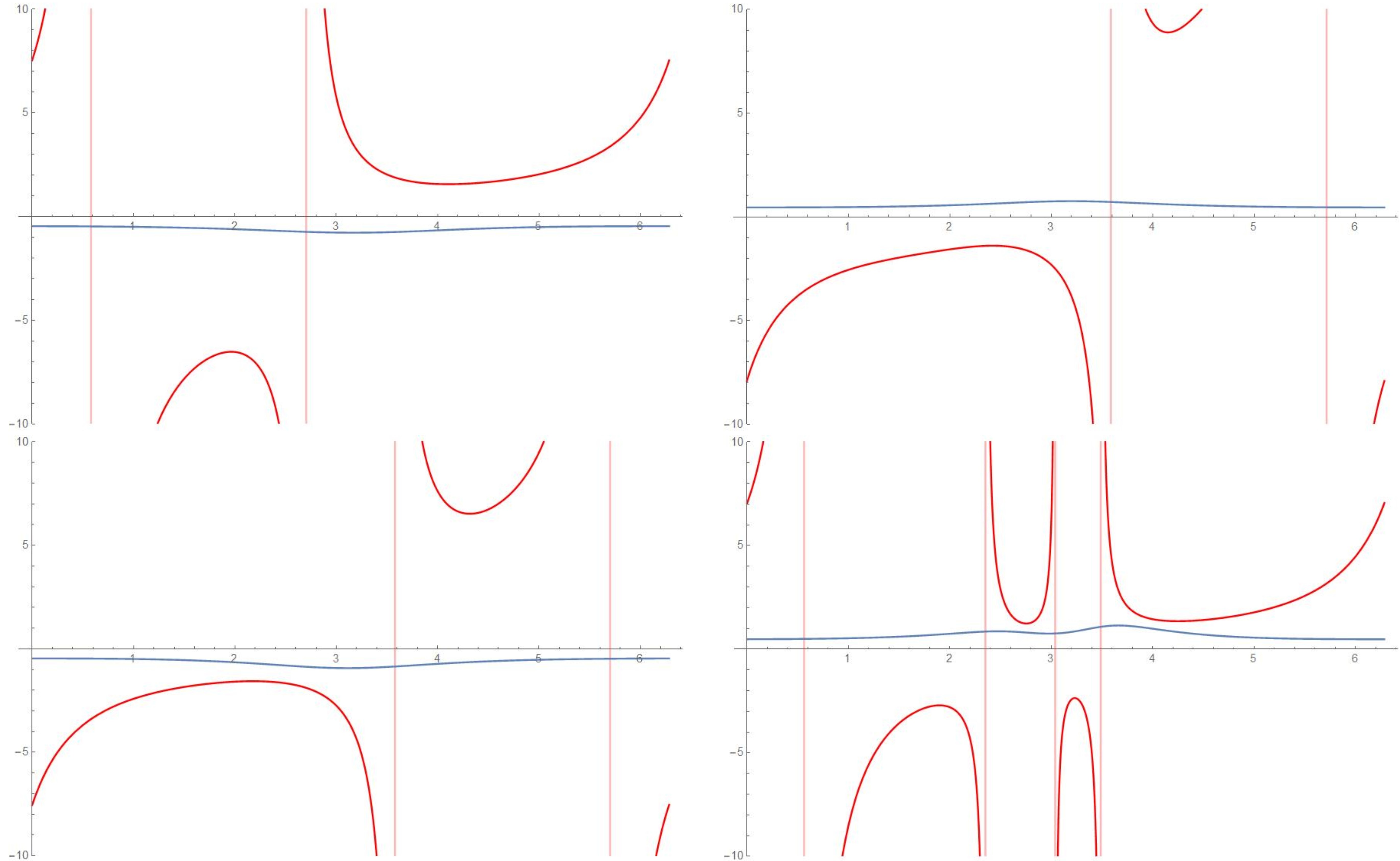}}};
    
    \pgfmathsetmacro{\SCALE}{0.9} 

    \node at ($\SCALE*(4.0,6.4)$) {$\bar{v}_{(1,+)}$};
    \node at ($\SCALE*(4.5,5.4)$) {$d_{(1,+)}$};
    \node at ($\SCALE*(8.9,6.2)$) {$d_{(2,-)}$};
    \node at ($\SCALE*(7.7,4.9)$) {$\bar{v}_{(2,-)}$};
    
    \node at ($\SCALE*(4.3,2.9)$) {$\bar{v}_{(1,-)}$};
    \node at ($\SCALE*(4.5,1.5)$) {$d_{(1,-)}$};
    \node at ($\SCALE*(10.7,2.45)$) {$\bar{v}_{(0,+)}$};
    \node at ($\SCALE*(8.0,2.3)$) {$d_{(0,+)}$};
    
\end{tikzpicture}
\caption{For the four developables of Fig.\ \ref{torus4povray}, comparison of the signed distance along the ruling between concentric circles $d_{(j,S)}$ and between the ridge and the regression curve $\bar{v}_{(j,S)}$. The surfaces are regular since the rulings reach the next foldline without first crossing the regression curve. We refer to the inner ridge as the $0$-th one.} \label{regVSvTorus}
\end{figure}

\begin{figure}
\begin{tikzpicture}
          \node[anchor=south west,inner sep=0] at (0,0) {\scalebox{1}{\includegraphics[width=\textwidth]{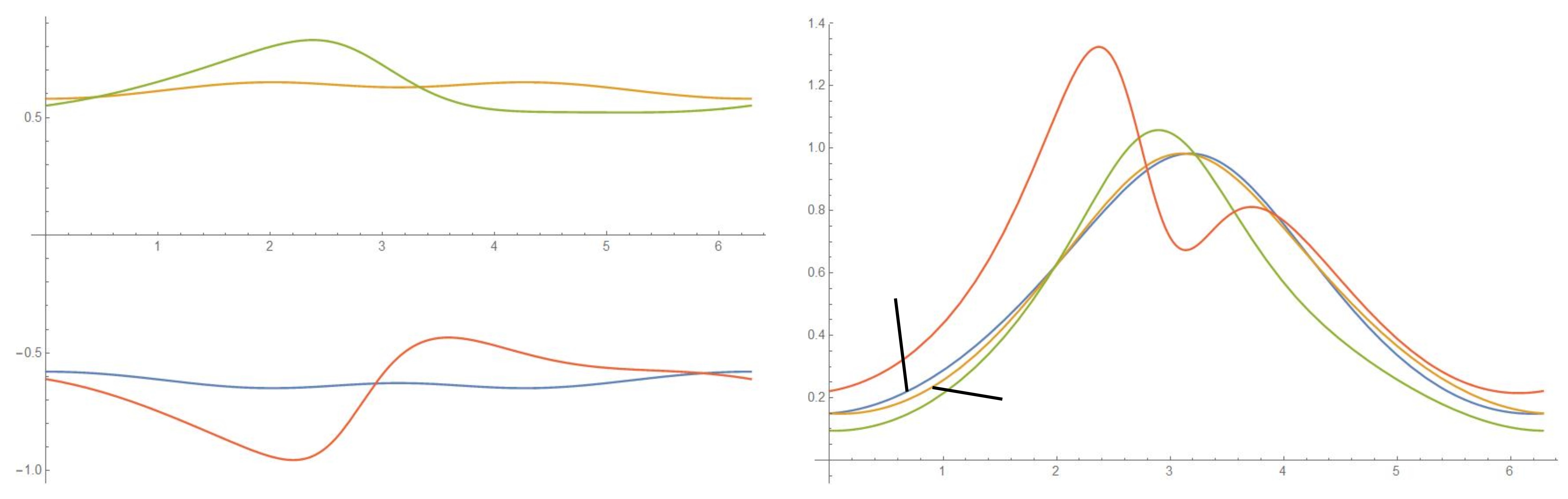}}};
    
    \pgfmathsetmacro{\SCALE}{1} 
    
    \node at ($\SCALE*(1.4,1.25)$) {$k_{(1,n,+)}$};
    \node at ($\SCALE*(2.3,3.1)$) {$k_{(1,n,-)}$};
    \node at ($\SCALE*(3.55,3.7)$) {$k_{(2,n,-)}$};
    \node at ($\SCALE*(3.4,0.4)$) {$k_{(0,n,+)}$};
    
     \node at ($\SCALE*(7.3,1.8)$) {$\tau_{(1,r,+)}$};
    \node at ($\SCALE*(10.45,1.0)$) {$\tau_{(2,r,-)}$};
    \node at ($\SCALE*(8.7,0.8)$) {$\tau_{(1,r,-)}$};
    \node at ($\SCALE*(7.95,3.1)$) {$\tau_{(0,r,+)}$};
    
\end{tikzpicture}
\caption{In the notation from \S \ref{secprop1}, normal curvature and relative torsion of the three ridges in the fold of Fig.\ \ref{torus4povray}. Note that $k_{(1,n,+)}$, $\tau_{(1,r,+)}$  and $k_{(1,n,-)}$, $\tau_{(1,r,-)}$ are respectively the normal curvature and relative torsion of the first ridge (the one from which the propagation starts) w.r.t. the developables on its two sides. We refer to the inner ridge as the $0$-th one.} \label{curvtorsTorus}
\end{figure}

\hfill

\end{document}